\documentclass{amsart}

\usepackage{mathtools,amssymb}
\usepackage[pdftex]{hyperref}
\usepackage{xcolor}
\usepackage{graphicx}
\usepackage{ytableau,verbatim,shuffle}
\usepackage{mathrsfs}
\usepackage{enumitem}
\usepackage[nameinlink,capitalise]{cleveref}
\usepackage{microtype}
\usepackage{hieroglf}
\usepackage{ifthen}

\usepackage{tikz}
\usetikzlibrary{calc, shapes, backgrounds,arrows,positioning,plotmarks}
\usetikzlibrary{arrows,positioning}
\usetikzlibrary{arrows.meta,tikzmark,decorations.markings,fit,calc,shapes,matrix,math}
\tikzset{
  flow/.style={->,thick},
  noflow/.style={-,gray},
 head/.style = {fill = white, text=black},
 plaque/.style = {draw, rectangle, minimum size = 10mm, fill=white}, 
  posplaque/.style = {draw, star,star points=7,star point ratio=0.8, minimum size = 10mm, fill=white}, 
}
\usetikzlibrary{decorations.pathmorphing}

\theoremstyle{plain}

\newtheorem{theorem}[equation]{Theorem}
\newtheorem{lemma}[equation]{Lemma}
\newtheorem{conjecture}[equation]{Conjecture}
\Crefname{conjecture}{Conjecture}{Conjectures}

\newtheorem{proposition}[equation]{Proposition}

\newtheorem{corollary}[equation]{Corollary}

\theoremstyle{definition}
\newtheorem{definition}[equation]{Definition}
\newtheorem{remark}[equation]{Remark}
\numberwithin{equation}{section}
\newtheorem{examplex}[equation]{Example}
\newenvironment{example}
  {\pushQED{\qed}\examplex}
  {\popQED\endexamplex}
\crefname{examplex}{example}{examples}

\usepackage{fullpage}

\tikzstyle{oriented}=[
postaction={decoration={markings,
	mark=at position 0.5 with {\arrow{#1}}},
	decorate}
]
\definecolor{amethyst}{rgb}{.6,0.2,1} 
\definecolor{dark}{rgb}{0.05, 0.5, 0.06} 

\newcommand{\bbb}{\mathsf{b}}

\newcommand{\trip}{\mathsf{trip}}
\newcommand{\prom}{\mathsf{prom}}

\DeclareMathOperator{\promotion}{\mathcal{P}}
\DeclareMathOperator{\gromotion}{\mathcal{G}}

\DeclareMathOperator{\tograph}{\mathfrak{G}}

\DeclareMathOperator{\lattice}{\mathcal{L}}
\newcommand{\SYT}{\mathsf{SYT}}

\newcommand{\cA}{\mathcal{A}}

\DeclareMathOperator{\grow}{\mathrm{grow}}

\newcommand{\inc}{\ensuremath{\mathrm{Inc}}}
\newcommand{\packinc}{\ensuremath{\mathrm{PackInc}}}

\newcommand{\nc}{\mathcal{NC}}

\newcommand{\SL}{\mathrm{SL}}

\newcommand{\fw}{\mathrm{FW}}

\definecolor{darkblue}{rgb}{0.0,0,0.7}

\newcommand{\newword}[1]{\textcolor{darkblue}{\textbf{\emph{#1}}}}

\AtBeginDocument{%
   \def\MR#1{}
}

\title{Promotion digraphs}
\author{Rebecca Patrias}
\address[RP]{\parbox{\linewidth}{Dept.\ of Mathematics, University of St.\ Thomas, St.\ Paul, MN, 55105, USA}}
\email{\parbox[t]{\linewidth}{rebecca.patrias@stthomas.edu}}
\author{Oliver Pechenik}
\address[OP]{\parbox{\linewidth}{Dept.\ of Combinatorics \& Optimization, University of Waterloo, Waterloo, ON, N2L 3G1, Canada}}
\email{\parbox[t]{\linewidth}{oliver.pechenik@uwaterloo.ca}}
\author{Jessica Striker}
\address[JS]{\parbox{\linewidth}{Dept.\ of Mathematics, North Dakota State University, Fargo, ND, 58102, USA}}
\email{\parbox[t]{\linewidth}{jessica.striker@ndsu.edu}}

\date{\today}


\begin{document}

\begin{abstract}
Work of Gaetz, Pechenik, Pfannerer, Striker, and Swanson (2024) introduced \emph{promotion permutations} for a rectangular standard Young tableau $T$. These promotion permutations encode important features of $T$ and its orbit under Sch\"utzenberger's promotion operator. Indeed, the promotion permutations uniquely determine the tableau $T$.

We introduce more general \emph{promotion digraphs} for both standard and increasing tableaux of arbitrary shape. For rectangular standard tableaux, this construction recovers the functional digraphs of the promotion permutations. Among other facts, we show that promotion digraphs uniquely determine $T$ when $T$ is standard of arbitrary shape. We show that an increasing tableau $T$ is not generally determined by its promotion digraphs but conjecture that it is when $T$ is rectangular. We provide evidence for this conjecture, including a complete characterization of the promotion digraphs for two-row rectangular increasing tableaux. We use promotion digraphs for three-row rectangular increasing tableaux to conjecture a connection between their dynamics and the \emph{flamingo webs} recently introduced by Kim to give a diagrammatic basis of the Specht module $S^{(k^3,1^{n-3k})}$.
\end{abstract}

\maketitle

\section{Introduction}

Standard Young tableaux are central objects in algebraic combinatorics. We are especially interested in the dynamics of M.-P.~Sch\"utzenberger's \emph{promotion} operator on tableaux, which has applications to both representation theory (e.g.\ \cite{Rhoades:thesis,Petersen.Pylyavskyy.Rhoades,Fontaine.Kamnitzer,Shen.Weng}) and algebraic geometry (e.g.\ \cite{Purbhoo,Speyer, Levinson}). To better understand the orbit structure of promotion, Gaetz, Pechenik, Pfannerer, Striker, and Swanson \cite{Gaetz.Pechenik.Pfannerer.Striker.Swanson:fluctuating} introduced
\emph{promotion permutations} for rectangular standard Young tableaux. These turned out to be key tools for developing the first rotation-invariant $SL_4$-web basis in \cite{Gaetz.Pechenik.Pfannerer.Striker.Swanson:SL4}. The key fact for standard tableaux of rectangular shape is that the order of promotion divides the number of values in the tableaux; this fact does not generalize to tableaux of general partition shape.

In \cite{Gaetz.Pechenik.Pfannerer.Striker.Swanson:fluctuating}, there is also a definition of \emph{promotion functions} for nonrectangular shapes, but almost nothing is studied about these functions outside of the rectangular case. Here, we first reinterpret promotion permutations and promotion functions by studying their functional digraphs, which we call \emph{promotion digraphs}. From this perspective, we prove a variety of results. In particular, we give a partial characterization of the digraphs that can arise as promotion digraphs and prove that any standard tableau $T$ is uniquely characterized by its promotion digraphs (\Cref{thm:antiexcedance}). We also take this opportunity to provide an additional exposition of some results of \cite{Gaetz.Pechenik.Pfannerer.Striker.Swanson:fluctuating} that were previously only described in the full generality of \emph{fluctuating tableaux}. This includes a useful characterization of promotion in terms of balance points (\Cref{prop:promotion_via_balance_points}), which has been used implicitly in the literature. We also define a minor variant of promotion, called \emph{gromotion}, as a convenient way to study an orbit of promotion.

The main reason that we prefer working with promotion digraphs over promotion functions is that we can also define promotion digraphs for \emph{$K$-promotion}, $\promotion$, on \emph{increasing tableaux}. An increasing tableau of shape $\lambda$ and entries in $[q] \coloneqq \{1,\ldots,q\}$ is a filling of the Young diagram of $\lambda$ with elements of $[q]$ such that rows and columns are strictly increasing; we write $\inc^q(\lambda)$ for the set of increasing tableaux of shape $\lambda$ with entries from $[q]$. $K$-promotion was introduced in \cite{Pechenik:CSP}, building on the combinatorics of increasing tableaux as developed by H.~Thomas and A.~Yong \cite{Thomas.Yong:K} in application to $K$-theoretic Schubert calculus. $K$-promotion has since received significant attention from the perspective of dynamical algebraic combinatorics. In contrast with standard tableaux, the order of $K$-promotion on increasing tableaux of rectangular shape $r \times c$ does not in general divide $q$ if $r\geq 4$, but can instead be significantly greater (see \Cref{ex:indegree0}). This behavior is only partially understood (see, in particular, \cite{Dilks.Pechenik.Striker,Patrias.Pechenik:CFDF}). However, it is known from \cite{Pechenik:CSP} that the order of $K$-promotion on increasing tableaux of rectangular shape $r \times c$ divides $q$ when $r \leq 2$. This leaves the case $r=3$ intriguingly open. In this case, we have the following conjecture but very few results.
\begin{conjecture}[\protect{\cite[Conjecture~4.12]{Dilks.Pechenik.Striker}}]\label{conj:3row_order}
    For $T \in \inc^q(3 \times c)$ a three-row rectangular increasing tableau, we have $\promotion^q(T) = T$.
\end{conjecture}
We hope that promotion digraphs can be a useful tool towards proving \Cref{conj:3row_order}. (\cite{Dilks.Pechenik.Striker} reports computer verification of \Cref{conj:3row_order} for $c \leq 7$; we have since additionally verified the conjecture for $c=8$.)

In this paper, we conjecture that a rectangular increasing tableau is uniquely determined by its promotion digraphs (\Cref{conj:inc_tab_determined}). A remarkable feature of this conjecture is that, unlike for standard tableaux, it does not extend to increasing tableaux of nonrectangular shapes; indeed, we give explicit examples of nonrectangular increasing tableaux with the same promotion digraphs (\Cref{ex:cant_get_shape,ex:cant_get_top_row}), and moreover of a nonrectangular increasing tableau with the same promotion digraph as a rectangular tableau (\Cref{ex:rect_nonrect_same_digraph}). Although we do not know a general algorithm for the reconstruction of a rectangular increasing tableau from its promotion digraphs, in practice our partial results are highly effective for reconstruction. We also give some partial characterizations of the digraphs that arise as promotion digraphs of increasing tableaux.

In the special case of two-row increasing tableaux, we prove an analogue (\Cref{thm:Kpromotion_2row_via_balance_points}) of the balance point characterization of promotion.
We use this characterization to give a complete description of the promotion digraphs of two-row rectangular increasing tableaux (\Cref{thm:complete_graphs}), establishing the $2$-row case of \Cref{conj:inc_tab_determined}. These promotion digraphs turn out to be closely related to the noncrossing set partitions that were connected to two-row rectangular increasing tableaux in \cite{Pechenik:CSP}. Reinterpreting noncrossing set partitions as a kind of \emph{plabic graph}, we show in \Cref{cor:trip=prom_2row} that these promotion digraphs coincide with a notion that we introduce (building on ideas of \cite{Postnikov, Gaetz.Pechenik.Pfannerer.Striker.Swanson:SL4}) of \emph{$(i,r)$-trip digraphs} for the corresponding plabic graphs. 
This result is an instance of the ``trip$=$prom phenomenon'' that is a guiding principle in recent work on webs (e.g.\ \cite{Hopkins.Rubey,Gaetz.Pechenik.Pfannerer.Striker.Swanson:SL4,Gaetz.Pechenik.Pfannerer.Striker.Swanson:2column}).

Building on \cite{Patrias.Pechenik.Striker,Fraser.Patrias.Pechenik.Striker}, J.~Kim \cite{Kim:embedding,Kim:flamingo} introduced \emph{flamingo webs} to give a diagrammatic web basis for the Specht module $S^{(k^3,1^{n-3k})}$ of the symmetric group $\mathfrak{S}_n$. We conjecture (\Cref{conj:trip=prom}) another instance of ``trip$=$prom'' in this context. More precisely, we claim that for each flamingo web $W$, there is a unique three-row rectangular increasing tableau $\tau(W) \in \inc^q(3 \times c)$ such that the $(i,3)$-trip digraphs of $W$ coincide with the promotion digraphs of $\tau(W)$. In particular, this would imply that all tableaux in the image of $\tau$ satisfy \Cref{conj:3row_order}. Note, however, that the map $\tau$ is not surjective; in ongoing work with Kim, we aim to develop generalized flamingo webs in ``trip$=$prom''-correspondence to the other three-row rectangular increasing tableaux. Moreover, the representation-theoretic significance of \Cref{conj:trip=prom} is currently mysterious, as is the appearance of plabic graphs in this story that is not facially related to positive geometries. Further discussion of these points will appear elsewhere.

\medskip
\noindent
{\bf This paper is organized as follows.} In \Cref{sec:SYT}, we study promotion digraphs of standard tableaux, including exposition of material from \cite{Gaetz.Pechenik.Pfannerer.Striker.Swanson:fluctuating} and our new results in the nonrectangular setting. Results about promotion digraphs for standard tableaux are summarized in \Cref{sec:summary_standard}. In \Cref{sec:increasing}, we turn to promotion digraphs for increasing tableaux. After some general constructions and results in \Cref{sec:general_inc,sec:inc_prom_digraph}, we further study the rectangular case in \Cref{sec:rect_inc,sec:vertex_degrees_inc}. 
The various results of \Cref{sec:general_inc,sec:inc_prom_digraph,sec:rect_inc,sec:vertex_degrees_inc} are briefly compiled and summarized in \Cref{sec:summary}. 
\Cref{sec:small_r} specializes to consider increasing tableaux with $2$ or $3$ rows, settings with special properties.
The case of two-row increasing tableaux is treated in detail in \Cref{sec:2row}. Finally, \Cref{sec:trips} recalls the theory of plabic graphs and flamingo webs, introduces their $(i,r)$-trip digraphs, and relates them to $K$-promotion on increasing tableaux with few rows via 
\Cref{cor:trip=prom_2row} and \Cref{conj:trip=prom}.

\section{Standard tableaux}\label{sec:SYT}
\subsection{Promotion and gromotion}

Consider an integer partition $\lambda=(\lambda_1\geq \lambda_2\geq\cdots\geq\lambda_r>0)$. The \newword{length} of $\lambda$ is $\ell(\lambda) \coloneqq r$. The \newword{Young diagram} corresponding to $\lambda$ is a left-justified array of boxes with $\lambda_i$ boxes in the $i$th row from the top. We often refer to both the integer partition and the Young diagram by $\lambda$. 
We let $|\lambda|$ denote the number of boxes in the shape $\lambda$.  
We further identify $\lambda$ with the poset whose elements are boxes of $\lambda$, whose unique minimal element is the upper left box, and whose cover relations are between adjacent boxes. 

An \newword{alphabet} $\cA$ is a totally ordered set. 
A \newword{standard Young tableau} of shape $\lambda$ and finite alphabet $\cA$ is a bijective filling of shape $\lambda$ with the elements of $\cA$ such that 
entries increase left-to-right across rows and down columns. 
We denote the set of standard Young tableaux of shape $\lambda$ with alphabet $\cA$ by $\SYT^{\cA}(\lambda)$. When $\cA=(1 < 2 < \dots < |\lambda|)$, we simply write $\SYT(\lambda)$. Additionally, let $\SYT^{\mathcal{A}}(r \times c)$ denote the set of standard Young tableaux of rectangular shape with $r$ rows and $c$ columns on the alphabet $\mathcal{A}$.
\begin{example}\label{ex:alphabet} 
Let
\[
T_1=\ytableaushort{12458,367}\quad \text{and} \quad  T_2=\ytableaushort{24678,3 \clubsuit 1}.
\]
Then 
$T_1\in\SYT((5,3))$ is a standard Young tableau of shape $(5,3)$ and alphabet $(1 < \dots < 8)$, while $T_2\in\SYT^{\mathcal{B}}((5,3))$ is a standard Young tableau of shape $(5,3)$ and alphabet $\mathcal{B} = (2 < 3 <4 < 6 <7 <8 < \clubsuit < 1)$.
\end{example}

For $T \in \SYT(\lambda)$, there is an associated \newword{lattice word} $\lattice(T)  = w_1 \dots w_{|\lambda|}$, where $w_i = j$ if the entry $i$ appears in row $j$ of $T$. 
The word $\lattice(T)$ is ``lattice'' in the sense that each initial segment contains at least as many instances of $j$ as of $j+1$, for each $j > 0$; note that every word with this property can be realized as $\lattice(U)$ for some standard tableau $U$. If every letter in a lattice word $\lattice$ appears an equal number of times, we say that $\lattice$ is \newword{balanced}. Note that $\lattice(U)$ is balanced if and only if $U$ is rectangular.

We next define an action on standard Young tableaux called \newword{promotion} \cite{Schutzenberger} via jeu de taquin slides. Further details can be found in, e.g., \cite{Rhoades:thesis, Stanley:PandE}. Given $T\in\SYT^\cA(\lambda)$ with $\cA = (a_1 < \dots < a_n)$ and $\bullet \notin \cA$, we form the promotion $\promotion(T)\in\SYT^\cA(\lambda)$ by doing the following:
\begin{enumerate}
    \item First, delete the unique entry $a_1$ in $T$, which is necessarily in the upper left corner. Call this box $\mathfrak{c}$ and fill it with the symbol $\bullet$.
    \item Now consider the boxes in $\lambda$ that are immediately right or below $\mathfrak{c}$, and denote by $\mathfrak{b}$ that with the smallest entry. Move the entry in $\mathfrak{b}$ into box $\mathfrak{c}$ and move the $\bullet$ into box $\mathfrak{b}$. Continue this procedure by now considering the boxes to the right of and below $\mathfrak{b}$, etc., until the box containing the $\bullet$ does not have any boxes of $\lambda$ to its right or below it. 
    \item Replace each entry $a_i$ with $a_{i-1}$ and replace the $\bullet$ with $a_n$.
\end{enumerate}

It will be useful to consider a minor variant of promotion, which we call \newword{gromotion}. This variant has been implicitly considered before; however, we find it valuable here to give it a distinct name for clarity. Given a finite alphabet $\cA = (a_1 < \dots < a_n)$, define $\grow(\cA)$ to be the finite alphabet $(a_2 < \dots < a_n < a_1)$ on the same underlying set. For $T \in \SYT^\cA(\lambda)$, we form the gromotion $\gromotion(T) \in \SYT^{\grow(\cA)}(\lambda)$ by following steps $(1)$ and $(2)$ of promotion, and then replacing step $(3)$ by the following:
\begin{enumerate}
    \item[$(3^*)$] Replace the $\bullet$ with $a_1$.
\end{enumerate}

\begin{example}\label{ex:nonrect_prom}
 Below we perform promotion on standard Young tableau $T$ with alphabet $\cA=(1<2<\dots < 9)$ to obtain $\promotion(T)\in \SYT(\lambda)$.
 \begin{align*}T=\ytableaushort{1346,259,78}\longrightarrow\ytableaushort{\bullet346,259,78}&\longrightarrow\ytableaushort{2346,\bullet59,78}\longrightarrow\ytableaushort{2346,5\bullet9,78}\longrightarrow\ytableaushort{2346,589,7\bullet}\\ &\longrightarrow\ytableaushort{1235,478,6\bullet}\longrightarrow \ytableaushort{1235,478,69}=\promotion(T)\end{align*}
 If we instead perform gromotion on $T$, we would arrive at $\gromotion(T)\in\SYT^{\grow(\cA)}(\lambda)$.

 \pushQED{\qed}
 \[\gromotion(T)=\ytableaushort{2346,589,71}\]  
 The lattice word for $T$ is $\lattice(T)=121121332$.
\end{example}

\begin{example}\label{ex:rect_prom}
We illustrate promotion on $U\in\SYT(3\times 3)$ with alphabet $\mathcal{B}=(1<2<\ldots<9)$ below, ending with $\promotion(U)\in\SYT(3\times 3)$. The result of gromotion $\gromotion(U)\in\SYT^{\grow(\mathcal{B})}(3\times 3)$ is shown below. We also compute $\lattice(U)=112321323$.
    \begin{align*}
U=\ytableaushort{126,358,479}&\longrightarrow\ytableaushort{\bullet26,358,479}\longrightarrow\ytableaushort{2\bullet6,358,479}\longrightarrow\ytableaushort{256,3\bullet8,479}\longrightarrow\ytableaushort{256,378,4\bullet9}\\ &\longrightarrow\ytableaushort{256,378,49\bullet}\longrightarrow\ytableaushort{145,267,38\bullet}\longrightarrow\ytableaushort{145,267,389}=\promotion({U})
    \end{align*}
    \pushQED{\qed}
    \[\gromotion(U)=\ytableaushort{256,378,491} \qedhere\]
    \let\qed\relax
\end{example}

\begin{definition}\label{def:prom_path}
    We define the \newword{flow path} of $T\in\SYT(\lambda)$ to be the set of boxes of $\lambda$ where the entries of $T$ and $\gromotion(T)$ differ.
\end{definition}
The flow path has in the past sometimes been called the \emph{jeu de taquin path}, the \emph{promotion path}, the \emph{sliding subposet}, or the \emph{Sch\"utzenberger path}. We prefer the phrase ``flow path'' (introduced in \cite{Pechenik:CSP}) in this context because we will be able to use the same name in the increasing tableaux setting in \Cref{sec:increasing}.
Note that by definition of gromotion, the flow path of a standard Young tableau is a maximal chain on the poset whose elements are boxes of $\lambda$ and whose unique minimal element is the upper left box; we will use this perspective in Definition~\ref{def:prom_path_inc} when generalizing to increasing tableaux, where this extra complication becomes necessary.

\begin{example}\label{ex:flowpath}
Below are the flow paths for the tableau $T$ from \Cref{ex:nonrect_prom} and the tableau $U$ from \Cref{ex:rect_prom}.
\[
\begin{ytableau}
*(gray) & & & \\
*(gray) & *(gray) & \\
\ & *(gray)
\end{ytableau}\hspace{1.5in}
\begin{ytableau}
   *(gray) & *(gray) & \\
   \  & *(gray) & \\
   \  & *(gray)& *(gray) \\
\end{ytableau} \qedhere
\]
\end{example}

We now describe the action of tableau promotion on the corresponding lattice words. Let $w = w_1 \ldots w_k$ be a lattice word. An \newword{$i$-balance point} is an index $j$ such that, in the subword $w_1 \ldots w_j$, the number of $i$'s equals the number of $(i+1)$'s. If we think of a standard tableau as a sequence of nested partitions, then an $i$-balance point $j$ corresponds to the $j$th partition having rows $i$ and $i+1$ of equal length.

The rectangular case of the following result is \cite[Proposition~4.9]{Gaetz.Pechenik.Pfannerer.Striker.Swanson:SL4}. 

\begin{proposition}\label{prop:balancing}
    Let $T \in \SYT(\lambda)$ and let $w = w_1 \ldots w_n = \lattice(T)$. Let the bottommost box of the flow path of $T$ lie in row $k + 1$. For $1 \leq i \leq k$, let $j_i \geq 2$ be the unique value such that promotion moves label $j_i$ from row $i+1$ to row $i$. Set $j_0 \coloneqq 1$. Then, for $i \leq k$, we have that $j_i$ is the first $i$-balance point of $w$ after $j_{i-1}$. In particular, the sequence
    \[
    1 < j_1 < j_2 < \dots < j_k
    \]
    is strictly increasing.
\end{proposition}
\begin{proof}
    This is straightforward from looking at the flow path of promotion. We have that $j_i$ is the entry of $T$ such that 
    \begin{itemize}
        \item $j_i$ appears in row $i+1$ of the flow path of $T$, and
        \item the box immediately above is also in the flow path of $T$. \qedhere
    \end{itemize} 
\end{proof}

\begin{example}\label{ex:balance}
 In \Cref{ex:rect_prom}, we have that $\lattice(U)=112321323$. We see that the first $1$-balance point is $5$ because the initial subword $11232$ has an equal number of $1$s and $2$s. We also see that $5$ is the label that moves from row $2$ to row $1$ when promoting $T$. The first $2$-balance point after $5$ is $7$ because $1123213$ is the first initial subword of $\lattice(U)$ of length greater than $5$ that has an equal number of $2$s and $3$s. We also see that entry $7$ moves from the third row to the second row when promoting $T$. Note that here $k=\ell(\lambda)-1=2$ because the flow path reaches the bottom row of the tableau (as indeed always happens for $\lambda$ rectangular); for an example where $k < \ell(\lambda) - 1$, see \Cref{ex:lattice_words_in_promotion} or consider the tableau $T_1$ from \Cref{ex:alphabet}.
\end{example}

The rectangular case of the following result is a special case of \cite[Proposition~4.10]{Gaetz.Pechenik.Pfannerer.Striker.Swanson:SL4} (and in different language \cite[Prop.~8.22]{Gaetz.Pechenik.Pfannerer.Striker.Swanson:fluctuating}).  We refer to it as the ``first balance point characterization'' of promotion, which is often implicitly used in the literature (see e.g.~\cite{Patrias:webs,Petersen.Pylyavskyy.Rhoades,Pfannerer.Rubey.Westbury,Tymoczko}). As this characterization is extremely useful, we take this opportunity to put a clear statement in the literature. This characterization will also be used in the proof of \Cref{prop:2rowrectdigraphs}.

\begin{proposition}\label{prop:promotion_via_balance_points}
    Let $T \in \SYT(\lambda)$ and let $w = w_1 \ldots w_n = \lattice(T)$. For $0 \leq i \leq k$, let $j_i$ be as in \Cref{prop:balancing}. Then $\lattice(\promotion(T))$ is obtained from $w$ by decrementing each entry $w_{j_i}$ by one, deleting the first letter of $w$, and appending a $k+1$ to the end.
\end{proposition}
\begin{proof}
    This is straightforward from \Cref{prop:balancing} and the definition of promotion. The appending of $k+1$ reflects that the flow path ended in row $k+1$. 
\end{proof}

\begin{example}\label{ex:lattice_words_in_promotion}
We continue our running example from \Cref{ex:balance}. To find  $\lattice(\promotion(U))$ from $\lattice(U)$, we first decrement $w_5$ and $w_7$, since $5$ is the first $1$-balance point and $7$ is the first $2$-balance point after $5$. (These balance points are shown in bold.) We then delete the $1$ at the beginning and append a $3$, since $k=2$ in the sequence $1<j_1<\cdots<j_k$ of balance points.
 \begin{align*}
     \lattice(U)&=1\ 1\ 2\ 3\ \mathbf{2}\ 1\ \mathbf{3}\ 2\ 3
     \\
     \lattice(\promotion(U))&=\ \,  \ 1 \ 2\ 3\ \mathbf{1}\ 1\ \mathbf{2}\ 2\ 3\ 3
 \end{align*}

We similarly illustrate the first-balance-point computation of several iterations of promotion with our tableau $T$ from \Cref{ex:nonrect_prom}. Note that the last letter of a lattice word is the number of bold entries of the prior lattice word plus one. 
\begin{align*}
    \lattice(T)&= 1\ \mathbf{2}\ 1\ 1\ 2\ 1\ 3\ \mathbf{3}\ 2\\
  \lattice(\promotion(T))&=\ \ \,  1\ 1\ 1\ 2\ 1\ 3\ 2\ 2\ 3\\ \lattice(\promotion^2(T))&=\ \ \, \ \ \,  1\ 1\ 2\ 1\ 3\ 2\ \mathbf{2}\ 3\ 1\\
\lattice(\promotion^3(T))&=\ \ \, \ \ \, \ \ \,  1\ \mathbf{2}\ 1\ \mathbf{3}\ 2\ 1\ 3\ 1\ 2\\
\lattice(\promotion^4(T))&=\ \ \, \ \ \, \ \ \, \ \ \, 1\ 1\ 2\ \mathbf{2}\ 1\ 3\ 1\ 2\ \mathbf{3}\\
\vdots \ \ \, \ \ &= \ \ \, \ \ \, \ \ \, \ \ \, \ \ \, \ \ \, \ \ \, \vdots \qedhere 
\end{align*} 
\end{example}

\subsection{Properties of promotion digraphs}\label{sec:standard_prom_digraph}

The following is derived from the standard tableau case of \cite[Proposition~6.9]{Gaetz.Pechenik.Pfannerer.Striker.Swanson:fluctuating}, which we modify to take as our definition. For a summary of the results on promotion digraphs proved in this and the next two subsections, see \Cref{thm:std_digraph_summary}.

\begin{definition}\label{def:promotion_digraph}
    Let $T \in \SYT(\lambda)$ with $|\lambda|=n$. For $i\geq 1$, we construct a directed graph $\prom_i(T)$ with vertex set $[n]$. Suppose that $\beta$ moves from row $i+1$ to $i$ in the application of gromotion to $\gromotion^{\alpha-1}(T)$. Then we draw an arrow $\alpha \to \beta$ in $\prom_i(T)$. We call $\prom_i(T)$ the \newword{$i$th promotion digraph} of the tableau $T$.
\end{definition}

Note that for $i \geq \ell(\lambda)$, \Cref{def:promotion_digraph} defines a graph that  has no edges.
Using gromotion instead of promotion simplifies the bookkeeping in \Cref{def:promotion_digraph}. If we instead used promotion, we would have an arrow $\alpha\to \alpha+\beta-1$ when $\beta$ moves from row $i+1$ to row $i$ in the application of promotion to $\promotion^{\alpha-1}(T)$; see \cite[Proposition~6.9]{Gaetz.Pechenik.Pfannerer.Striker.Swanson:fluctuating}.

\begin{example}\label{ex:nonrect_digraph}
    For $T$ the tableau from \Cref{ex:nonrect_prom}, we have that 
    \[\prom_1(T) = \raisebox{-1.1cm}{\begin{tikzpicture}[node distance=1cm]
    \node (1) at (0, 2) {1};
    \node (9) at (-1, 2) {9};
    \node (2) at (1, 2) {2};
    \node (3) at (-2, 2) {3};
    \node (6) at (-1, 1) {6};
    \node (7) at (0, 1) {7};
    \node (4) at (-1, 0) {4};
    \node (5) at (0, 0) {5};
    \node (8) at (1, 0) {8};
    \draw[->, thick] (3) -- (9);
    \draw[->, thick] (9) -- (1);
    \draw[->, thick] (1) -- (2);
    \draw[->, thick] (6) -- (7);
    \draw[->, thick] (4) -- (5);
    \draw[<->, thick] (5) -- (8); 
\end{tikzpicture}}
\quad \text{and} \quad
\prom_2(T) = \raisebox{-1.1cm}{\begin{tikzpicture}[node distance=1cm]
    \node (1) at (0, 2) {1};
    \node (6) at (-1, 2) {6};
    \node (8) at (1, 2) {8};
    \node (9) at (-2, 2) {9};
    \node (4) at (-1, 1) {4};
    \node (7) at (0, 1) {7};
    \node (2) at (1, 1) {2};
    \node (3) at (-1, 0) {3};
    \node (5) at (0, 0) {5};
    \draw[->, thick] (9) -- (6);
    \draw[->, thick] (6) -- (1);
    \draw[->, thick] (1) -- (8);
    \draw[->, thick] (4) -- (7);
\end{tikzpicture}}
\]

Below is the gromotion orbit used to compute these promotion digraphs above. Note that we keep track of the alphabet for each tableau by considering all bold numbers to be larger than all non-bold numbers. 
\pushQED{\qed}
\ytableausetup{centertableaux}
\begin{align*}\ytableaushort{1346,259,78}&\rightarrow\ytableaushort{2346,589,7{\mathbf{1}}}\rightarrow\ytableaushort{346{\mathbf{2}},589,7{\mathbf{1}}}\rightarrow\ytableaushort{469{\mathbf{2}},58{\mathbf{3}},7{\mathbf{1}}}\rightarrow\ytableaushort{569{\mathbf{2}},78{\mathbf{3}},{\mathbf{1}}{\mathbf{4}}}\rightarrow\ytableaushort{689{\mathbf{2}},7{\mathbf{3}}{\mathbf{5}},{\mathbf{1}}{\mathbf{4}}}\\
    &\rightarrow\ytableaushort{789{\mathbf{2}},{\mathbf{1}}{\mathbf{3}}{\mathbf{5}},{\mathbf{4}}{\mathbf{6}}}\rightarrow\ytableaushort{89{\mathbf{2}}{\mathbf{7}},{\mathbf{1}}{\mathbf{3}}{\mathbf{5}},{\mathbf{4}}{\mathbf{6}}}\rightarrow \ytableaushort{9{\mathbf{2}}{\mathbf{5}}{\mathbf{7}},{\mathbf{1}}{\mathbf{3}}{\mathbf{8}},{\mathbf{4}}{\mathbf{6}}}\rightarrow \ytableaushort{{\mathbf{1}}{\mathbf{2}}{\mathbf{5}}{\mathbf{7}},{\mathbf{3}}{\mathbf{6}}{\mathbf{8}},{\mathbf{4}}{\mathbf{9}}} \qedhere\end{align*}
\end{example}

\begin{example}\label{ex:rect_digraph}
    For $U$ the rectangular tableau from \Cref{ex:rect_prom}, we have that 
    \[
    \prom_1(U) = 
\raisebox{-2.1cm}
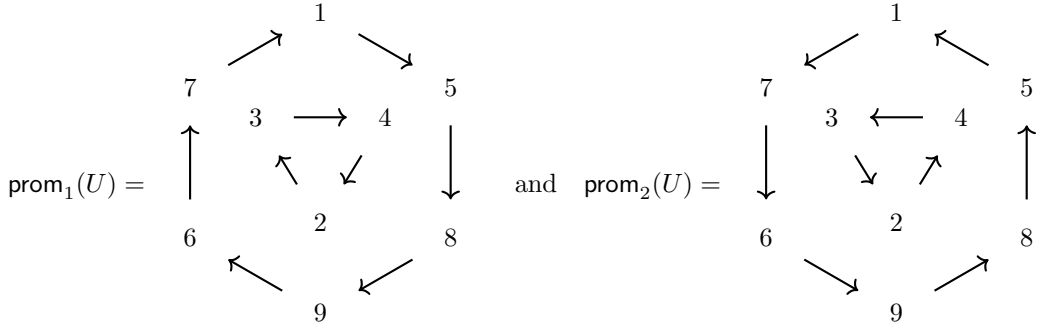
{\begin{tikzpicture}[scale=2, every node/.style={minimum size=10mm}]
    \node (1) at (90:1) {1};   
    \node (5) at (30:1) {5};  
    \node (8) at (-30:1) {8};  
    \node (9) at (-90:1) {9};  
    \node (6) at (-150:1) {6}; 
    \node (7) at (150:1) {7};  
    \node (2) at (0,-0.4) {2};
    \node (3) at (-0.43, 0.3) {3};
    \node (4) at (0.43, 0.3) {4};
    \draw[->, thick] (1) -- (5);
    \draw[->, thick] (5) -- (8);
    \draw[->, thick] (8) -- (9);
    \draw[->, thick] (9) -- (6);
    \draw[->, thick] (6) -- (7);
    \draw[->, thick] (7) -- (1);
    \draw[->, thick] (2) -- (3);
    \draw[->, thick] (3) -- (4);
    \draw[->, thick] (4) -- (2);
\end{tikzpicture}}
    \quad \text{and} \quad
    \prom_2(U) = 
    \raisebox{-2.1cm}
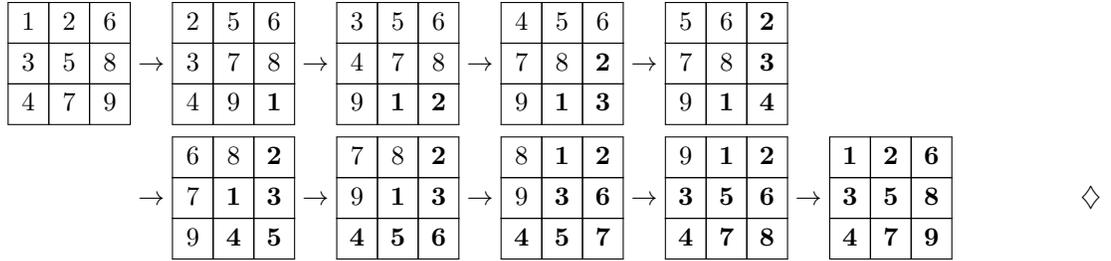
{\begin{tikzpicture}[scale=2, every node/.style={minimum size=10mm}]
  \node (1) at (90:1) {1};   
    \node (5) at (30:1) {5};  
    \node (8) at (-30:1) {8};  
    \node (9) at (-90:1) {9};  
    \node (6) at (-150:1) {6}; 
    \node (7) at (150:1) {7};  
    \node (2) at (0,-0.4) {2};
    \node (3) at (-0.43, 0.3) {3};
    \node (4) at (0.43, 0.3) {4};
    \draw[<-, thick] (1) -- (5);
    \draw[<-, thick] (5) -- (8);
    \draw[<-, thick] (8) -- (9);
    \draw[<-, thick] (9) -- (6);
    \draw[<-, thick] (6) -- (7);
    \draw[<-, thick] (7) -- (1);
    \draw[<-, thick] (2) -- (3);
    \draw[<-, thick] (3) -- (4);
    \draw[<-, thick] (4) -- (2);
\end{tikzpicture}}
    \]
\pushQED{\qed}

    \begin{align*}
    \ytableausetup{centertableaux}
\ytableaushort{126,358,479}&\rightarrow\ytableaushort{256,378,49{\mathbf{1}}}\rightarrow\ytableaushort{356,478,9{\mathbf{1}}{\mathbf{2}}}\rightarrow\ytableaushort{456,78{\mathbf{2}},9{\mathbf{1}}{\mathbf{3}}}\rightarrow\ytableaushort{56{\mathbf{2}},78{\mathbf{3}},9{\mathbf{1}}{\mathbf{4}}}\\ &\rightarrow\ytableaushort{68{\mathbf{2}},7{\mathbf{1}}{\mathbf{3}},9{\mathbf{4}}{\mathbf{5}}}\rightarrow\ytableaushort{78{\mathbf{2}},9{\mathbf{1}}{\mathbf{3}},{\mathbf{4}}{\mathbf{5}}{\mathbf{6}}}\rightarrow\ytableaushort{8{\mathbf{1}}{\mathbf{2}},9{\mathbf{3}}{\mathbf{6}},{\mathbf{4}}{\mathbf{5}}{\mathbf{7}}}\rightarrow\ytableaushort{9{\mathbf{1}}{\mathbf{2}},{\mathbf{3}}{\mathbf{5}}{\mathbf{6}},{\mathbf{4}}{\mathbf{7}}{\mathbf{8}}}\rightarrow\ytableaushort{{\mathbf{1}}{\mathbf{2}}{\mathbf{6}},{\mathbf{3}}{\mathbf{5}}{\mathbf{8}},{\mathbf{4}}{\mathbf{7}}{\mathbf{9}}} \qedhere
    \end{align*}
\end{example}
\ytableausetup{nocentertableaux}

We now describe some properties of promotion digraphs.

\begin{proposition}\label{prop:no_loops}
 Let $T \in \SYT(\lambda)$. Then the digraph $\prom_i(T)$ has no loops, and each vertex has outdegree either $0$ or $1$.

\end{proposition}
\begin{proof}
    In the construction of promotion digraphs,  during the $j$th application of gromotion, there is either a unique element of the alphabet that moves from row $i+1$ to $i$ or there is no such element. Hence, there is at most one arrow in $\prom_i(T)$ with tail at vertex $j$. If this arrow exists, its head cannot also be at vertex $j$, for at the $j$th application of gromotion, the tableau has label $j$ in its upper left corner; clearly, the label $j$ cannot move from row $i+1$ to row $i$, since it is in row $1$.
\end{proof}

\begin{corollary} 
\label{cor:std_fixed_pt_free}
For $T\in \SYT(\lambda)$ with $|\lambda| = n$, the promotion digraph
$\prom_i(T)$ is a fixed-point-free partial function $\prom_i(T) : [n] \to [n]$.  \qed
\end{corollary}

We have the following relation between the promotion digraph of $T$ and the promotion digraphs of $\promotion(T)$, which tells us partial information about how to create $\prom_i(\promotion(T))$ from $\prom_i(T)$.

\begin{lemma}[{\cite[Lemma~6.4]{Gaetz.Pechenik.Pfannerer.Striker.Swanson:fluctuating}}]\label{lem:prom.P}
  Let $T \in \SYT(\lambda)$ with $|\lambda| = n$. Then we have
    \[ \prom_i(T)(\alpha+1) = \beta+1 \qquad\Longleftrightarrow\qquad \prom_i(\promotion(T))(\alpha) = \beta \]
  for all $1 \leq \alpha,\beta \leq n-1$ and all $i \geq 1$. 
\end{lemma}

\begin{example}
    Let $T$ be the tableau from \Cref{ex:nonrect_digraph}. Then we can compute that
     \[\prom_1(\promotion(T)) = \raisebox{-1.1cm}{\begin{tikzpicture}[node distance=1cm]
    \node (1) at (0, 2) {9};
    \node (9) at (-1, 2) {8};
    \node (2) at (1, 2) {1};
    \node (3) at (-2, 2) {2};
    \node (6) at (-1, 1) {5};
    \node (7) at (0, 1) {6};
    \node (4) at (-1, 0) {3};
    \node (5) at (0, 0) {4};
    \node (8) at (1, 0) {7};
    \draw[->, thick] (3) -- (9);
    \draw[->, thick] (9) -- (1);
    \draw[->, thick] (6) -- (7);
    \draw[->, thick] (4) -- (5);
    \draw[<->, thick] (5) -- (8); 
\end{tikzpicture}}
\quad \text{and} \quad
\prom_2(\promotion(T)) = \raisebox{-1.1cm}{\begin{tikzpicture}[node distance=1cm]
    \node (1) at (0, 2) {9};
    \node (6) at (-1, 2) {5};
    \node (8) at (1, 2) {7};
    \node (9) at (-2, 2) {8};
    \node (4) at (-1, 1) {3};
    \node (7) at (0, 1) {6};
    \node (2) at (1, 1) {1};
    \node (3) at (-1, 0) {2};
    \node (5) at (0, 0) {4};
    \draw[->, thick] (9) -- (6);
    \draw[->, thick] (6) -- (1);
    \draw[->, thick] (4) -- (7);
\end{tikzpicture}}, 
\]
while we also have
    \[\prom_1(\promotion^2(T)) = \raisebox{-1.1cm}{\begin{tikzpicture}[node distance=1cm]
    \node (1) at (0, 2) {8};
    \node (9) at (-1, 2) {7};
    \node (2) at (1, 2) {9};
    \node (3) at (-2, 2) {1};
    \node (6) at (-1, 1) {4};
    \node (7) at (0, 1) {5};
    \node (4) at (-1, 0) {2};
    \node (5) at (0, 0) {3};
    \node (8) at (1, 0) {6};
    \draw[->, thick] (3) -- (9);
    \draw[->, thick] (9) -- (1);
     \draw[->, thick] (2.north) to [out=150,in=30] (3.north);
    \draw[->, thick] (6) -- (7);
    \draw[->, thick] (4) -- (5);
    \draw[<->, thick] (5) -- (8); 
\end{tikzpicture}}
\quad \text{and} \quad
\prom_2(\promotion^2(T)) = \raisebox{-1.1cm}{\begin{tikzpicture}[node distance=1cm]
    \node (1) at (0, 2) {8};
    \node (6) at (-1, 2) {4};
    \node (8) at (1, 2) {6};
    \node (9) at (-2, 2) {7};
    \node (4) at (-1, 1) {2};
    \node (7) at (0, 1) {5};
    \node (2) at (1, 1) {9};
    \node (3) at (-1, 0) {1};
    \node (5) at (0, 0) {3};
    \draw[->, thick] (9) -- (6);
    \draw[->, thick] (6) -- (1);
     \draw[->, thick] (2.north) to [out=150,in=30] (4.north);
    \draw[->, thick] (4) -- (7);
\end{tikzpicture}}.
\]
The reader should compare these digraphs to those in \Cref{ex:nonrect_digraph}. 
\end{example}

\begin{remark}
\Cref{lem:prom.P} says that we decrement all vertex labels and keep the outgoing arrows from all vertices except the vertex labeled $1$ before decrementing; the lemma does not say anything about what outgoing arrows there should be from this vertex in the new digraph. 
    Note that \Cref{lem:prom.P} does not make any claims in the case where $\alpha$ or $\beta$ equals $n$; see the above example. We will see in \Cref{cor:rect_rotate} (where $n = cr$) that whenever $T$ is rectangular, we have an analogous property for $\alpha$ or $\beta$ equaling $n$. This gives us complete information about how to construct $\prom_i(\promotion(T))$ from $\prom_i(T)$ for rectangular $T$.
\end{remark}

We now describe properties of promotion digraphs that determine the entries of the corresponding tableau.

\begin{definition}\label{def:excedance}
    Let $\mathcal{D}$ be a digraph on vertex set $[n]$. A number $j \in [n]$ is an \newword{excedance} of $\mathcal{D}$ if there is an edge $i \to j$ for some $i < j$; similarly, $j$ is an \newword{antiexcedance} if there is an edge $i \to j$ for some $i > j$. We say $j$ is a \newword{nonexcedance} if it is not an excedance. 
\end{definition} 

Note that a value $j$ can be simultaneously an excedance and an antiexcedance.

\begin{theorem}
\label{thm:antiexcedance}
If $T \in \SYT(\lambda)$ and $i \geq 1$, 
then the nonexcedances of $\prom_i(T)$ are exactly the numbers in the first $i$ rows of $T$.
\end{theorem}
\begin{proof}
Fix $i$ and $T$. Consider an entry $\beta$ in $T$. After $\beta-1$ steps of gromotion, the entry $\beta$ will occupy the upper left corner of the tableau. In particular, if $\beta$ begins in row $R> i$, then at some point during the first $\beta-1$ steps of gromotion, $\beta$ moves from row $i+1$ to row $i$. If this happens during step $\alpha \leq \beta-1$, then $\prom_i(T)$ has a directed edge $\alpha \to \beta$ and so $\beta$ is an excedance.

If instead $\beta$ begins in row $R' \leq i$, then $\beta$ remains weakly north of row $i$ throughout the first $\beta-1$ steps of gromotion. Hence, $\prom_i(T)$ has no directed edge $\alpha \to \beta$ for any $\alpha < \beta$, and so $\beta$ is a nonexcedance.
\end{proof}

\begin{remark}
The rectangular case of \Cref{thm:antiexcedance} coincides with the standard tableau case of \cite[Theorem 6.12]{Gaetz.Pechenik.Pfannerer.Striker.Swanson:fluctuating}, which is phrased in terms of {antiexcedances} rather than nonexcedances. In the rectangular case, these notions are equivalent since the promotion digraph is the functional digraph of a fixed-point-free permutation \cite[Theorem 6.7]{Gaetz.Pechenik.Pfannerer.Striker.Swanson:fluctuating}. 
\end{remark}

\begin{corollary}\label{cor:uniquely_determined}
    Suppose $T\in\SYT(\lambda)$. Then $T$ is uniquely determined by its set of promotion digraphs. \qed
\end{corollary}

The reader can see \Cref{ex:nonrect_digraph} and \Cref{ex:rect_digraph} as illustrations of the previous corollary. Note that in \Cref{ex:nonrect_digraph}, $1,3,4,$ and $6$ are the nonexcedances of $\prom_1(T)$, and so comprise the first row of $T$. All the vertices except $7$ and $8$ are nonexcedances of $\prom_2(T)$ and are the content of the first two rows of $T$.

\subsection{Rectangular tableaux}
\label{sec:rect_std}
It is well known that the $n$th power of promotion on rectangular standard Young tableaux with $n$ cells is the identity (see, e.g., \cite{Haiman,Stanley:PandE}).

\begin{theorem}\label{thm:ft.prom_evac}
Let $T \in \SYT(r \times c)$ be a rectangular standard Young tableau. Then
$\promotion^{cr}(T) = T$. 
\end{theorem}

The following is the main result of  \cite{Gaetz.Pechenik.Pfannerer.Striker.Swanson:fluctuating}. Let $\sigma = (1\,2\,\cdots\,cr)$ be the long cycle in the symmetric group $\mathfrak{S}_{cr}$.

\begin{theorem}[{\cite[Theorem~6.7 \& Corollary~6.8]{Gaetz.Pechenik.Pfannerer.Striker.Swanson:fluctuating}}]\label{thm:prom_perms}
  Let $T \in \SYT(r \times c)$. Then for all $1 \leq i \leq r-1$:
  \begin{enumerate}
    \item $\prom_i(T)$ is a permutation,
    \item $\prom_i(T) = \prom_{r-i}(T)^{-1}$,
    \item $\prom_i(\promotion(T)) = \sigma^{-1} \prom_i(T) \sigma$,
  \end{enumerate}
  Moreover, if $r$ is even, then $\prom_{r/2}(T)$ is an involution. \qed
\end{theorem}

The following is a more explicit version of \Cref{thm:prom_perms}(3). See \Cref{ex:rect_digraph} for an illustration.
\begin{corollary}\label{cor:rect_rotate}
 For $T \in \SYT(r \times c)$, we have
 \[ \prom_i(T)(\alpha+1) = \beta+1 \qquad\Longleftrightarrow\qquad \prom_i(\promotion(T))(\alpha) = \beta \]
  for all $1 \leq \alpha, \beta \leq cr$ and all $1 \leq i \leq r-1$. (Here, we interpret the number $cr+1$ as $1$.)
  \end{corollary}
  \begin{proof}
      For $\alpha,\beta \neq cr$, this is the rectangular case of \Cref{lem:prom.P}. The remaining cases then follow from \Cref{thm:prom_perms}(1).
  \end{proof}

\begin{corollary}\label{cor:uniquely_determined_rect}
    Suppose $T\in\SYT(r \times c)$ is rectangular. Then $T$ is uniquely determined by its promotion permutations $(\prom_i(T))_{i=1}^{\lfloor r/2 \rfloor}$.
\end{corollary}
\begin{proof}
    This follows by combining \Cref{thm:prom_perms} with \Cref{cor:uniquely_determined}.
\end{proof}

However, we have the following in the case that $T$ has at most two rows. (This was claimed without proof in \cite{Gaetz.Pechenik.Pfannerer.Striker.Swanson:SL4}.)

\begin{definition}
    A \newword{noncrossing matching} of $[n]=\{1,\ldots,n\}$ is a partition of $[n]$ into pairs $(a,b)$ with $a<b$ such that there are no pairs $(i, k)$ and $(j,\ell)$ such that $i<j<k<\ell$.
\end{definition}
 We often draw a noncrossing matching as a graph with vertices $1,\ldots,n$ (drawn linearly or circularly) and edges between paired vertices (drawn so that none of the edges intersect).

By \Cref{thm:prom_perms}, $\prom_1$ of a $2$-row rectangular tableau is a fixed-point-free involution. In fact, we can say something much stronger.

\begin{proposition}[cf.\ {\cite[Theorem~9.3 ($r=2$)]{Gaetz.Pechenik.Pfannerer.Striker.Swanson:SL4}}]\label{prop:2rowrectdigraphs}
If $T \in \SYT(2\times c)$, then $\prom_1(T)$ consists of exactly $c$ $2$-cycles that form a noncrossing matching. (For $i > 1$, $\prom_i(T)$ has no edges.) Moreover, every noncrossing matching of $[2c]$  arises as $\prom_1(U)$ for some $U \in \SYT(2 \times c)$.
\end{proposition}
\begin{proof}
Let $w = \lattice(T)$. We construct a noncrossing matching of $[2c]$  by repeatedly reading $w$ from left to right and pairing any $b_1$ with $b_2$ if $w_{b_1} = 1$, $w_{b_2} = 2$ and all numbers between $b_1$ and $b_2$ are already paired. 
 This produces the usual noncrossing matching $\tograph(T)$ associated with the $2$-row standard Young tableau (see, for example, \cite{Stanley:Catalan,Tymoczko,Patrias.Pechenik:evacuation}). Note that this operation is invertible; given a noncrossing matching on $2c$ vertices, we recover the lattice word by reading a $1$ for each vertex that is paired with a larger indexed vertex and reading a $2$ for each vertex paired with a smaller indexed vertex. 

It was observed by Dennis White (see e.g.~\cite[Proof of Theorem~8.2]{Rhoades:thesis}) that this bijection intertwines promotion with the rotation of noncrossing matchings. Our proof below recovers this fact in the context of promotion digraphs.

 As in \Cref{prop:balancing}, let $j_1$ be the unique entry of $T$ that moves from row $2$ to row $1$ when applying promotion to $T$. Recall from \Cref{prop:promotion_via_balance_points}, that $\lattice(\promotion(T))$ is obtained from $w$ by decrementing $w_{j_1}$, deleting the first letter, and appending $2$ to the end. Note also that $\prom_1(T)(1) = j_1$ and that $j_1$ is matched with $1$ in $\tograph(T)$.

 Consider the noncrossing matching $\tograph(T)$ as a graph with vertices arranged clockwise around a circle. If we rotate $\tograph(T)$ counterclockwise by an angle of $2\pi/{2c}$, the effect on the corresponding lattice word is to relabel vertex $j_1$ from $2$ to $1$, move the first vertex to the end, relabel it from $1$ to $2$. This matches the effect on lattice words of applying promotion, so we see that the bijection $\tograph$ intertwines promotion with rotation. By $\prom_1(T)(1) = j_1$ and this promotion equivariance, we then obtain that $\prom_1(T)(m)$ is the number paired with $m$ in $\tograph(T)$. Thus $\prom_1(T)$ consists of $2$-cycles where the two elements are paired in the noncrossing matching $\tograph(T)$.
\end{proof}

\begin{example}\label{ex:2rowrectdigraphs}
Let
\begin{center}
    $T=$ \ytableaushort{1246,3578}.\\
    Then $\lattice(T)=$
\raisebox{-.11cm}{\begin{tikzpicture}[node distance=1cm]
    \node (1) at (0, 0) {1};
    \node (2) at (1,0) {1};
    \node (3) at (2,0) {2};
    \node (4) at (3,0) {1};
    \node (5) at (4,0) {2};
    \node (6) at (5,0) {1};
    \node (7) at (6,0) {2};
    \node (8) at (7,0) {2};
     \draw[-] (2.north) to [out=30,in=150] (3.north);
          \draw[-] (4.north) to [out=30,in=150] (5.north);
            \draw[-] (6.north) to [out=30,in=150] (7.north);
             \draw[-] (1.north) to [out=30,in=150] (8.north);
\end{tikzpicture}}\\

\vspace{.2in}
and
$\prom_1(T):$
\raisebox{-0.6cm}{\begin{tikzpicture}[node distance=1cm]
    \node (1) at (0, 0) {1};
    \node (2) at (1,0) {8};
    \node (3) at (2,0) {2};
    \node (4) at (3,0) {3};
    \node (5) at (0,1) {4};
    \node (6) at (1,1) {5};
    \node (7) at (2,1) {6};
    \node (8) at (3,1) {7};
          \draw[<->, thick] (1) to (2);
            \draw[<->, thick] (3) to (4);
             \draw[<->, thick] (5) to (6);
             \draw[<->, thick] (7) to (8);
\end{tikzpicture}}
\end{center}
\end{example}

We do not know of a characterization of the sets of promotion digraphs that can arise from standard Young tableaux $T$ in general, even in the case that $T$ has two rows. 
\begin{example}
Let $T =  \ytableaushort{134,25}$. Then iterated gromotion gives
\begin{align*}
  \ytableaushort{134,25} \rightarrow \ytableaushort{234,5{\mathbf{1}}} \rightarrow \ytableaushort{34{\mathbf{2}},5{\mathbf{1}}}\rightarrow\ytableaushort{4{\mathbf{1}}{\mathbf{2}},5{\mathbf{3}}}\rightarrow
  \ytableaushort{5{\mathbf{1}}{\mathbf{2}},{\mathbf{3}}{\mathbf{4}}}\rightarrow\ytableaushort{{\mathbf{1}}{\mathbf{2}}{\mathbf{5}},{\mathbf{3}}{\mathbf{4}}},
  \end{align*}
so that
  \[\prom_1(T) = \raisebox{-.6cm}{\begin{tikzpicture}[node distance=1cm]
    \node (1) at (0, 2) {2};
    \node (9) at (-1, 2) {1};
    \node (3) at (-2, 2) {3};
    \node (6) at (-1, 1) {4};
    \node (7) at (0, 1) {5};
    \draw[->, thick] (3) -- (9);
    \draw[->, thick] (9) -- (1);
    \draw[->, thick] (6) -- (7);
\end{tikzpicture}}.\]
Note that the promotion digraph for this nonrectangular tableau is very far from being a noncrossing matching.
\end{example}

\subsection{Summary of the preceeding results}\label{sec:summary_standard}

Here, we summarize the various results obtained regarding the promotion digraphs of standard tableaux in \Cref{sec:standard_prom_digraph,sec:rect_std}. 

\begin{theorem}
\label{thm:std_digraph_summary}
Let $T\in\SYT(\lambda)$, where $|\lambda|=n$ and $\lambda$ has $r$ rows. We have the following.

\begin{enumerate}
\item Each digraph $\prom_i(T)$ has no loops, and each vertex has outdegree either $0$ or $1$. 
\item $\prom_i(T)$ is a fixed-point-free partial function $\prom_i(T) : [n] \to [n]$. 
\item  $\prom_i(T)(\alpha+1) = \beta+1 \Longleftrightarrow \prom_i(\promotion(T))(\alpha) = \beta$
  for all $1 \leq \alpha,\beta \leq n-1$ and all $1\leq i \leq r-1$. 
\item The nonexcedances of $\prom_i(T)$ are exactly the numbers in the first $i$ rows of $T$. 
\end{enumerate}

Now assume that $\lambda=r\times c$ is rectangular. Then we additionally have the following.

\begin{enumerate}[resume]
\item $\prom_i(T)$ is a permutation. 
\item $\prom_i(T) = \prom_{r-i}(T)^{-1}$.
\item
$\prom_i(\promotion(T)) = \sigma^{-1} \prom_i(T) \sigma$.
\item If $r$ is even, then $\prom_{r/2}(T)$ is an involution. 
\item $\prom_i(T)(\alpha+1) = \beta+1 \Longleftrightarrow \prom_i(\promotion(T))(\alpha) = \beta$
  for all $1 \leq \alpha, \beta \leq cr$ and all $1 \leq i \leq r-1$. (Here, we interpret the number $cr+1$ as $1$.) 
\item $T$ is uniquely determined by its promotion permutations $(\prom_i(T))_{i=1}^{\lfloor r/2 \rfloor}$. 
\item  When $r=2$, $\prom_1(T)$ consists of exactly $c$ $2$-cycles that form a noncrossing matching. 
Moreover, every noncrossing matching of $[2c]$  arises as $\prom_1(U)$ for some $U \in \SYT(2 \times c)$. 

\end{enumerate}
\end{theorem}
\begin{proof}
    This theorem collects together the statements of \Cref{prop:no_loops,cor:std_fixed_pt_free,lem:prom.P,thm:antiexcedance,thm:prom_perms,cor:rect_rotate,cor:uniquely_determined_rect,prop:2rowrectdigraphs}. Note that these include both new results and results from prior papers, as noted in each statement.
\end{proof}

\section{Increasing tableaux}\label{sec:increasing}
\subsection{\texorpdfstring{$K$}{K}-promotion and gromotion}\label{sec:general_inc}
Let $\mathcal{A}$ be a finite alphabet.
An \newword{increasing tableau} of shape $\lambda$ and alphabet $\cA$ is a filling of the Young diagram $\lambda$ with elements of $\mathcal{A}$ such that rows are strictly increasing left-to-right and columns are strictly increasing top-to-bottom. We say an increasing tableau $T$ on the alphabet $\mathcal{A}$ is \newword{packed} if every element of $\mathcal{A}$ appears in $T$. 
Let $\inc^\mathcal{A}(\lambda)$ denote the set of increasing tableaux of shape $\lambda$ on the alphabet $\mathcal{A}$ and let $\packinc^\mathcal{A}(\lambda)$ denote the subset of packed tableaux.
We write $\inc^q(\lambda)$ when $\mathcal{A}$ is the alphabet $\{1, \dots, q\}$ with its standard ordering.

Suppose $\lambda$ has $r$ rows. Let $2^{[r]}$ be the collection of all subsets of $[r]$. The \newword{lattice word} of an increasing tableau $T \in \inc^q(\lambda)$ is a word $\lattice(T) = w_1 \ldots w_{q}$ in the alphabet $2^{[r]}$ where $w_i$ is the set of rows in which $T$ has an entry $i$. If we write the elements of each subset in some order, then the resulting word in the alphabet $[r]$ is a lattice word.

\begin{example}\label{ex:inc_latticeword}
    Let 
    \[
    T = \ytableaushort{1 2 3 5, 2 4 5 8, 4 6 7 9, 6 8 {10} {11}}.
    \]
    Then 
    \[
    \lattice(T) =w =  1 \{1, 2\} 1 \{2, 3\} \{1, 2\} \{3, 4\} 3 \{2, 4\} 3 4 4
    \]
    (where we have suppressed brackets around singleton sets for better readability). If we write the elements of each subset in decreasing order, we obtain the word $1211322143342344$, which is a lattice word. Choosing different orders for the elements of subsets, we can instead obtain $1121321243324344$, which is also a lattice word.
\end{example}

We next recall the definition of \newword{$K$-promotion} \cite{Pechenik:CSP} for  increasing tableaux $T$ on the alphabet $\mathcal{A}$. 
Given $T\in\inc^\cA(\lambda)$ with $\cA = (a_1 < \dots < a_n)$ and $\bullet \notin \cA$, we form the $K$-promotion $\promotion(T)\in\inc^\cA(\lambda)$ by doing the following:
\begin{enumerate}
        \item  First, if the entry $a_1$ exists in $T$ (necessarily in the upper left corner), delete it and fill its box with the symbol $\bullet$.
        Set $i=2$.
        \item Let $\mathfrak{B}$ be the set (possibly empty) of boxes with $\bullet$ in them. 
        Let $\mathfrak{C}$ be the set of boxes that have entry $a_i$ and are adjacent to a box in $\mathfrak{B}$. Fill all boxes in $\mathfrak{C}$ with $\bullet$ and all boxes in $\mathfrak{B}$ that are adjacent to a box in $\mathfrak{C}$ with $a_i$. If $i = n$, proceed to the next step; otherwise, increment $i$ and repeat this step.      
    \item Replace each entry $a_j$ with $a_{j-1}$ and replace all $\bullet$s with $a_n$.
\end{enumerate}

As before, \newword{$K$-gromotion} $\gromotion(T) \in \inc^{\grow(\cA)}(\lambda)$ is the minor variant of $K$-promotion where we replace step (3) by 
\begin{enumerate}
    \item[$(3^*)$] Replace all $\bullet$s with $a_1$.
\end{enumerate}

\begin{example}\label{ex:KproandKgro}
\ytableausetup{boxsize=1.3em,centertableaux}
Below, we perform $K$-promotion on an increasing tableau $T \in \inc^\cA(5,4,3)$ with alphabet $\cA = (1 < 2 < \dots <  9)$ to obtain $\promotion(T) \in \inc^\cA(5,4,3)$. 
 \begin{align*}   
T&=\ytableaushort{12357,2369,468}\rightarrow\ytableaushort{\bullet2357,2369,468}\rightarrow\ytableaushort{2\bullet357,\bullet369,468}\rightarrow\ytableaushort{23\bullet57,3\bullet69,468}\rightarrow \ytableaushort{235\bullet7,3\bullet69,468}\\ &\rightarrow\ytableaushort{235\bullet7,36\bullet9,4\bullet8}\rightarrow\ytableaushort{2357\bullet,36\bullet9,4\bullet8}\rightarrow\ytableaushort{2357\bullet,3689,48\bullet}\rightarrow\ytableaushort{1246\bullet,2578,37\bullet}\rightarrow\ytableaushort{12469,2578,379}=\promotion(T)
 \end{align*}
 If instead we had performed $K$-gromotion, we would have obtained
 $\gromotion(T)\in \inc^{\grow(\cA)}(5,4,3)$ below.
     \[\gromotion(T)=\ytableaushort{23571,3689,481} \qedhere\]
\ytableausetup{boxsize=normal,nocentertableaux}
\end{example}

For $\bbb$ a box of a Young diagram, we write $\bbb^\downarrow$ for the box immediately below $\bbb$ and $\bbb^\rightarrow$ for the box immediately right of $\bbb$.

 We say that $\beta$ \newword{moves from row $i+1$ to $i$} during $K$-gromotion of $T$ if during an iteration of step (2), there is a box $\bbb \in \mathfrak{B}$ of row $i$ such that $\bbb$ gets filled with $\beta$ and $\bbb^\downarrow \in \mathfrak{C}$ gets filled with $\bullet$. For instance, in \Cref{ex:KproandKgro}, $2$ and $3$ move from row $2$ to row $1$, $6$ and $8$ move from row $3$ to row $2$. If $\beta$ moves from row $i+1$ to $i$ for some $i$, we say that $\beta$ \newword{moves up}.

As in the previous section, we define the flow path of an increasing tableau. 
\begin{definition}\label{def:prom_path_inc}
    We define the \newword{flow path} of $T\in\inc^q(\lambda)$ to be the subposet $P$ of $\lambda$ 
    constructed as follows: \begin{itemize}
    \item If $T$ has no entry $1$, $P$ is the empty poset.
    \item Otherwise, begin to build $P$ by including the minimal element of $\lambda$ (which has entry $1$). For each box $\bbb$ in $P$, include the cover relation $\bbb < \bbb'$ and the box $\bbb'$ where $\bbb' \in \{\bbb^\rightarrow, \bbb^\downarrow \}$ is the adjacent box to the right or below with the smallest entry; if both entries are equal, include both elements $\bbb^\rightarrow, \bbb^\downarrow$ and both cover relations.  Repeat this process until no new elements or cover relations are added.
    \end{itemize}
\end{definition}

\begin{example}
The increasing tableau $U$ below has flow path shown.
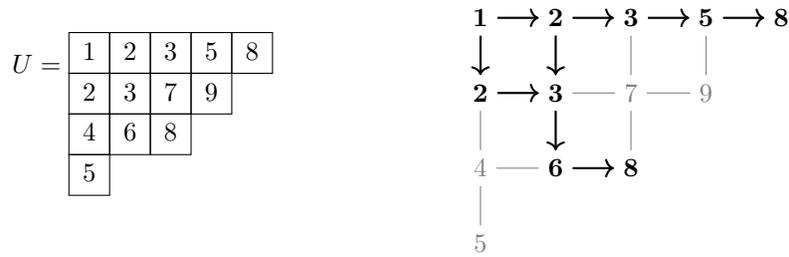
\[U= \ytableaushort{12358,2379,468,5}\hspace{1in}
    \raisebox{-2.5cm}{\begin{tikzpicture}
        \node (1) at (0,0) {$\mathbf{1}$};
        \node (2) at (1,0) {$\mathbf{2}$};
        \node (3) at (2,0) {$\mathbf{3}$};
        \node (5) at (3,0) {$\mathbf{5}$};
        \node (7) at (4,0) {$\mathbf{8}$};
        \node (2second) at (0,-1) {$\mathbf{2}$};
        \node (3second) at (1,-1) {$\mathbf{3}$};
        \node (6) at (2,-1) {\textcolor{gray}{7}};
        \node (9) at (3,-1) {\textcolor{gray}{9}};
        \node (4) at (0,-2) {\textcolor{gray}{4}};
        \node (6third) at (1,-2) {$\mathbf{6}$};
        \node (8) at (2,-2) {$\mathbf{8}$};
        \node (5fourth) at (0,-3) {\textcolor{gray}{5}};
        \draw[flow] (1) to (2);
        \draw[flow] (1) to (2second);
        \draw[noflow] (2) to (3) to (5) to (7);
        \draw[noflow] (2second) to (3second) to (6) to (9);
        \draw[noflow] (4) to (6third) to (8);
        \draw[noflow] (2second) to (4);
        \draw[-] (2) to (3second) to (6third);
        \draw[noflow] (3) to (6) to (8);
        \draw[noflow] (5) to (9);
        \draw[noflow] (4) to (5fourth);
        \draw[flow] (2) to (3);
        \draw[flow] (2) to (3second);
        \draw[flow] (2second) to (3second);
        \draw[flow] (3) to (5);
        \draw[flow] (5) to (7);
        \draw[flow] (3second) to (6third);
        \draw[flow] (6third) to (8);
    \end{tikzpicture}}\]
    Note that in this case, the flow path is not an induced subposet of $\lambda$. However, flow paths of increasing tableaux often are induced subposets, and those of standard tableaux always are.
\end{example}

\begin{example}
Recall the tableau $T$ from \Cref{ex:inc_latticeword}. Its flow path is shown to its right. This example shows that the flow path is not determined by its set of boxes; all the boxes are in the flow path, but this does not indicate which arrows are included.
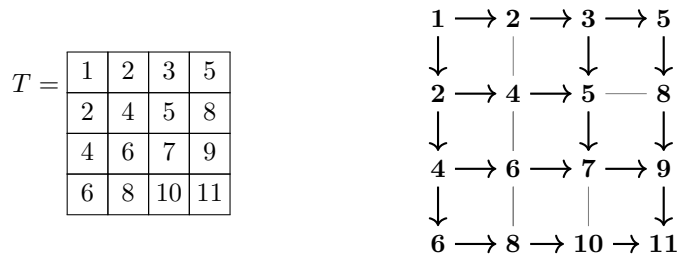
\[
    T = \ytableaushort{1 2 3 5, 2 4 5 8, 4 6 7 9, 6 8 {10} {11}}\hspace{1in}
    \raisebox{-2.25cm}{\begin{tikzpicture}
        \node (1) at (0,0) {$\mathbf{1}$};
        \node (2) at (1,0) {$\mathbf{2}$};
        \node (3) at (2,0) {$\mathbf{3}$};
        \node (5) at (3,0) {$\mathbf{5}$};
        \node (2second) at (0,-1) {$\mathbf{2}$};
        \node (4) at (1,-1) {$\mathbf{4}$};
        \node (5second) at (2,-1) {$\mathbf{5}$};
        \node (6) at (3,-1) {$\mathbf{8}$};
        \node (4third) at (0,-2) {$\mathbf{4}$};
        \node (6third) at (1,-2) {$\mathbf{6}$};
        \node (7) at (2,-2) {$\mathbf{7}$};
        \node (9) at (3,-2) {$\mathbf{9}$};
        \node (6fourth) at (0,-3) {$\mathbf{6}$};
        \node (8fourth) at (1,-3) {$\mathbf{8}$};
        \node (10) at (2,-3) {$\mathbf{10}$};
        \node (11) at (3,-3) {$\mathbf{11}$};
        \draw[noflow] (1) to (2) to (3) to (5) to (6) to (5second) to (4) to (2second) to (1) to (2second) to (4third) to (6third) to (4) to (6third) to (7) to (5second) to (7) to (9) to (6);
        \draw[noflow] (4third) to (6fourth) to (8fourth) to (6third) to (8fourth) to (10) to (7) to (9) to (11) to (10);
        \draw[noflow] (2) to (4) to (5second) to (3);
        \draw[flow] (1) to (2);
        \draw[flow] (2) to (3);
        \draw[flow] (3) to (5);
        \draw[flow] (3) to (5second);
        \draw[flow] (1) to (2second);
        \draw[flow] (2second) to (4);
        \draw[flow] (4) to (5second);
        \draw[flow] (5) to (6);
        \draw[flow] (2second) to (4third);
        \draw[flow] (4third) to (6third);
        \draw[flow] (4third) to (6fourth);
        \draw[flow] (5second) to (7);
        \draw[flow] (6third) to (7);
        \draw[flow] (7) to (9);
        \draw[flow] (9) to (11);
        \draw[flow] (6) to (9);
        \draw[flow] (6fourth) to (8fourth);
        \draw[flow] (8fourth) to (10);
        \draw[flow] (10) to (11);
    \end{tikzpicture}}
    \]
\end{example}

\begin{proposition}
    The flow path of $T\in\inc^q(\lambda)$ is the poset whose elements are all boxes that contain $\bullet$ at some point during promotion and whose cover relations are of the form $\bbb< \bbb^\downarrow$ whenever the entry in $\bbb^\downarrow$ moves up and $\bbb < \bbb^\rightarrow$ whenever the entry in $\bbb^\rightarrow$ moves left.
\end{proposition}
\begin{proof}
    This follows from step (2) of the definition of $K$-promotion; during this step, $\bullet$ is filled with the smallest entry in an adjacent box. 
\end{proof}

We now introduce some definitions that will be useful in \Cref{sec:2row}. Although these definitions make sense for arbitrary shapes $\lambda$, an analogue of \Cref{prop:promotion_via_balance_points} is not available in this generality. We will however prove an analogue for the case $\lambda = (\lambda_1, \lambda_2)$ in \Cref{thm:Kpromotion_2row_via_balance_points}. We give the general definition here as we expect it to be useful in future investigations.

\begin{definition}\label{def:Kbalance_point}
    Let $T \in \inc^q(\lambda)$ with lattice word $\lattice(T) = w = w_1 \ldots w_q$. An \newword{$i$-balance point} of $w$ is a $j$ such that there is an equal number of $i$'s and $(i+1)$'s appearing in the subsets $w_1, w_2, \dots, w_j$. We say that $j$ is an \newword{$i$-teetering point} of $w$ if $i$ and $(i+1)$ both appear in $w_j$ and the number of $i$'s appearing in the subsets $w_1, w_2, \dots, w_j$ is exactly one more than the number of $(i+1)$'s.
\end{definition}

\begin{remark}\label{rem:balance_and_teetering_on_the_tab}
We can also locate $i$-balance points and $i$-teetering points of $w=\lattice(T)$ by looking at $T$. Here, $j$ is an $i$-teetering point if $j$ appears in box $\bbb$ of row $i+1$ and also appears in box $\bbb^{\uparrow\rightarrow}$ of row $i$.
Similarly, $j$ is an $i$-balance point if $j$ appears in box $\bbb$ of row $i+1$ and the box $\bbb^{\uparrow\rightarrow}$ either does not exist or contains some $j' > j$.
\end{remark}

\begin{remark}
If we replace $\lattice(T) = w$ with a word $w'$ in the alphabet $[r]$ by writing the elements of each subset in increasing order, then the balance points of $w'$ correspond to the balance points of $w$. If we instead produce a word $w''$ by writing the elements of each subset in decreasing order, then the balance points of $w''$ correspond to both the balance and teetering points of $w$.
\end{remark}

\begin{example}\label{ex:inc_balance}
Recall the lattice word
    \[
    \lattice(T) =w =  1 \{1, 2\} 1 \{2, 3\} \{1, 2\} \{3, 4\} 3 \{2, 4\} 3 4 4
    \]
    from \Cref{ex:inc_latticeword}. 
    There is a $1$-balance point of $w$ at $8$ because there are four instances of $2$ and four instances of $1$ among the sets $w_1, \dots, w_8$. The reader may check that there is no earlier $1$-balance point; however, $2$ and $5$ are $1$-teetering points. 
    
    The first $2$-balance point is $7$ with an earlier $2$-teetering point $4$ and a later $2$-balance point $9$. The only $3$-balance point is $11$, but there is also a $3$-teetering point $6$.  

    The reader may also check the characterization of balance and teetering points from \Cref{rem:balance_and_teetering_on_the_tab} against the tableau illustrated in \Cref{ex:inc_latticeword}.
\end{example}

\subsection{Promotion digraphs of increasing tableaux}
\label{sec:inc_prom_digraph}
We now introduce promotion digraphs for increasing tableaux, extending the promotion digraphs for standard tableaux defined in \Cref{sec:standard_prom_digraph}. For a summary of the results on promotion digraphs proved in this and the next two subsections, see \Cref{thm:inc_digraph_summary}.
\begin{definition}
        Let $T \in \inc^q(\lambda)$. For $i\geq 1$, we construct a directed graph $\prom_i(T)$ with vertex set $[q]$. Suppose that $\beta$ moves from row $i+1$ to $i$ in the application of gromotion to $\gromotion^{\alpha-1}(T)$. Then we draw an arrow $\alpha \to \beta$ in $\prom_i(T)$. We call $\prom_i(T)$ the \newword{$i$th promotion digraph} of the increasing tableau $T$.
\end{definition}

\begin{example}\label{ex:inc_promotion_digraph}
    Let $T \in \inc^9(5,4,3)$ be as in \Cref{ex:KproandKgro}. Gromotion proceeds as follows.

    \begin{align*}
    &\ytableaushort{12357,2369,468}\rightarrow\ytableaushort{2357{\mathbf{1}},3689,48{\mathbf{1}}}\rightarrow\ytableaushort{3579{\mathbf{1}},468{\mathbf{2}},8{\mathbf{1}}{\mathbf{2}}}\rightarrow\ytableaushort{4579{\mathbf{1}},68{\mathbf{2}}{\mathbf{3}},8{\mathbf{1}}{\mathbf{3}}}\rightarrow\ytableaushort{579{\mathbf{1}}{\mathbf{4}},68{\mathbf{2}}{\mathbf{3}},8{\mathbf{1}}{\mathbf{3}}}
        \rightarrow\\ &\ytableaushort{679{\mathbf{1}}{\mathbf{4}},8{\mathbf{1}}{\mathbf{2}}{\mathbf{3}},{\mathbf{1}}{\mathbf{3}}{\mathbf{5}}}\rightarrow\ytableaushort{79{\mathbf{1}}{\mathbf{3}}{\mathbf{4}},8{\mathbf{1}}{\mathbf{2}}{\mathbf{6}},{\mathbf{1}}{\mathbf{3}}{\mathbf{5}}}\rightarrow\ytableaushort{89{\mathbf{1}}{\mathbf{3}}{\mathbf{4}},{\mathbf{1}}{\mathbf{2}}{\mathbf{5}}{\mathbf{6}},{\mathbf{3}}{\mathbf{5}}{\mathbf{7}}}\rightarrow\ytableaushort{9{\mathbf{1}}{\mathbf{3}}{\mathbf{4}}{\mathbf{8}},{\mathbf{1}}{\mathbf{2}}{\mathbf{5}}{\mathbf{6}},{\mathbf{3}}{\mathbf{5}}{\mathbf{7}}}
        \rightarrow \ytableaushort{{\mathbf{1}}{\mathbf{2}}{\mathbf{3}}{\mathbf{4}}{\mathbf{8}},{\mathbf{2}}{\mathbf{5}}{\mathbf{6}}{\mathbf{9}},{\mathbf{3}}{\mathbf{7}}{\mathbf{9}}}
    \end{align*}

Thus we have the following promotion digraphs.
\[
\raisebox{-3cm}{
    \begin{tikzpicture}[scale=2, every node/.style={minimum size=3mm}]
    \node (label) at (-1.4,-0.6) {$\prom_1(T) =$};
    \node (1) at (0.4, -0.2) {1};   
    \node (2) at (1.2, -0.2) {2}; 
    \node (3) at (0.8, -1) {3};  
    \node (4) at (1.6, -1) {4};   
    \node (5) at (-0.8, -1) {5};    
    \node (6) at (-0.0, -1) {6};  
    \node (7) at (0, -1.5) {7};     
    \node (8) at (0.8, -1.5) {8};  
    \node (9) at (0.8, 0.6) {9};   
    \draw[->, thick] (1) -- (2);   
    \draw[->, thick] (1) -- (3);   
    \draw[->, thick] (2) -- (3);   
    \draw[->, thick] (3) -- (4);   
    \draw[->, thick] (9) -- (1);   
    \draw[->, thick] (2) to (9);   
    \draw[->, thick] (9) to (2);   
    \draw[->, thick] (5) -- (6);   
    \draw[->, thick] (6) -- (3);   
    \draw[->, thick] (7) -- (8);   
\end{tikzpicture}}\hspace{.2in}
\raisebox{-3cm}{
    \begin{tikzpicture}[scale=2, every node/.style={minimum size=3mm}]
    \node (label) at (-1.4,-0.6) {$\prom_2(T) =$};
    \node (1) at (1.1, -0.5) {1};   
    \node (7) at (0.7, 0) {7}; 
    \node (8) at (0.7, -1) {8};  
    \node (9) at (-0.5, -0.5) {9};   
    \node (5) at (0.3, -0.5) {5};    
    \node (6) at (1.9, -0.5) {6};  
    \node (3) at (0, -1.5) {3};     
    \node (2) at (0.8, -1.5) {2};  
    \node (4) at (1.6, -1.5) {4};   
    \draw[->, thick] (9) -- (5);
    \draw[->, thick] (5) -- (1);
    \draw[->, thick] (1) -- (6);
    \draw[->, thick] (7) -- (5);   
    \draw[->, thick] (7) -- (1);  
    \draw[->, thick] (5) -- (8);  
    \draw[->, thick] (1) -- (8);     
    \draw[->, thick] (3) -- (2);   
    \draw[->, thick] (2) -- (4);   
\end{tikzpicture}\qedhere}
\]
\end{example}

\begin{example}\label{ex:inc_promotion_digraph_rectangle}
    Consider $U\in\inc^{10}(3\times 4)$ shown below. \[U=\ytableaushort{1236,4569,789{10}}\]
    We compute $10$ iterations of gromotion on $U$ and use this to construct $\prom_1(U)$ and $\prom_2(U)$. 
    \begin{align*}
        \ytableaushort{1236,4569,789{10}}&\rightarrow \ytableaushort{2369, 459{10},78{10}{\mathbf{1}}}\rightarrow\ytableaushort{3569,489{10},7{10}{\mathbf{1}}{\mathbf{2}}}\rightarrow\ytableaushort{4569,789{10},{10}{\mathbf{1}}{\mathbf{2}}{\mathbf{3}}}\rightarrow\ytableaushort{569{10},78{10}{\mathbf{3}},{10}{\mathbf{1}}{\mathbf{2}}{\mathbf{4}}}\rightarrow\ytableaushort{689{10},7{10}{\mathbf{2}}{\mathbf{3}},{10}{\mathbf{1}}{\mathbf{4}}{\mathbf{5}}}\\ &\rightarrow \ytableaushort{789{10},{10}{\mathbf{1}}{\mathbf{2}}{\mathbf{3}},{\mathbf{1}}{\mathbf{4}}{\mathbf{5}}{\mathbf{6}}}\rightarrow\ytableaushort{89{10}{\mathbf{3}},{10}{\mathbf{1}}{\mathbf{2}}{\mathbf{6}},{\mathbf{1}}{\mathbf{4}}{\mathbf{5}}{\mathbf{7}}}\rightarrow\ytableaushort{9{10}{\mathbf{2}}{\mathbf{3}},{10}{\mathbf{1}}{\mathbf{5}}{\mathbf{6}},{\mathbf{1}}{\mathbf{4}}{\mathbf{7}}{\mathbf{8}}}\rightarrow\ytableaushort{{10}{\mathbf{1}}{\mathbf{2}}{\mathbf{3}},{\mathbf{1}}{\mathbf{4}}{\mathbf{5}}{\mathbf{6}},{\mathbf{4}}{\mathbf{7}}{\mathbf{8}}{\mathbf{9}}}\rightarrow\ytableaushort{{\mathbf{1}}{\mathbf{2}}{\mathbf{3}}{\mathbf{6}},{\mathbf{4}}{\mathbf{5}}{\mathbf{6}}{\mathbf{9}},{\mathbf{7}}{\mathbf{8}}{\mathbf{9}}{\mathbf{10}}}
    \end{align*}
  \[ \raisebox{.5in}{ $\prom_1(U)=$}\begin{tikzpicture}
        \node (1) at (0,0) {$1$};
        \node (6) at (1,0) {$6$};
        \node (7) at (2,0) {$7$};
        \node (9) at (1, -1) {$9$};
        \node (10) at (0,-2) {$10$};
        \node (4) at (1,-2) {$4$};
        \node (3) at (2,-2) {$3$};
        \node (2) at (3,-.5) {$2$};
        \node (5) at (2.5,-1.5) {$5$};
        \node (8) at (3.5,-1.5) {$8$};
        \draw[->, thick] (1) to (6);
        \draw[->, thick] (6) to (7);
        \draw[->, thick] (7) to (3);
        \draw[->, thick] (1) to (9);
        \draw[->, thick] (9) to (1);
        \draw[->, thick] (9) to (10);
        \draw[->, thick] (4) to (9);
        \draw[->, thick] (4) to (10);
        \draw[->, thick] (3) to (4);
        \draw[->, thick] (10) to (1);
        \draw[->, thick] (10.north) to [out=140,in=140, looseness=1.8] (6.north);
        \draw[->, thick] (5) to (8);
        \draw[->, thick] (8) to (2);
        \draw[->, thick] (2) to (5);
    \end{tikzpicture}\hspace{.3in}
\raisebox{.5in}{$\prom_2(U)=$}
\begin{tikzpicture}
        \node (1) at (0,0) {$1$};
        \node (6) at (1,0) {$6$};
        \node (7) at (2,0) {$7$};
        \node (9) at (1, -1) {$9$};
        \node (10) at (0,-2) {$10$};
        \node (10b) at (-.1,-1.8) {};
        \node (4) at (1,-2) {$4$};
        \node (3) at (2,-2) {$3$};
        \node (2) at (3,-.5) {$2$};
        \node (5) at (2.5,-1.5) {$5$};
        \node (8) at (3.5,-1.5) {$8$};
        \draw[<-, thick] (1) to (6);
        \draw[<-, thick] (6) to (7);
        \draw[<-, thick] (7) to (3);
        \draw[<-, thick] (1) to (9);
        \draw[<-, thick] (9) to (1);
        \draw[<-, thick] (9) to (10);
        \draw[<-, thick] (4) to (9);
        \draw[<-, thick] (4) to (10);
        \draw[<-, thick] (3) to (4);
        \draw[<-, thick] (10) to (1);
        \draw[<-, thick] (10b.north) to [out=140,in=140, looseness=1.8] (6.north);
        \draw[<-, thick] (5) to (8);
        \draw[<-, thick] (8) to (2);
        \draw[<-, thick] (2) to (5);
    \end{tikzpicture} \qedhere\]
\end{example}

We next establish some properties of promotion digraphs of increasing tableaux.
\begin{lemma}\label{lem:inc_no_loops}
 Let $T \in \inc^q(\lambda)$. Then the digraph $\prom_i(T)$ has no loops.
\end{lemma}
\begin{proof}
 If an arrow in $\prom_i(T)$ with tail at vertex $j$ exists, its head cannot also be at vertex $j$. This is because immediately before the $j$th application of gromotion, the tableau has a label $j$ in its upper left corner and no label $j$ in any other part of the tableau. Clearly, the label $j$ cannot move from row $i+1$ to row $i$, since the only label $j$ is in row $1$.
\end{proof}

Compare \Cref{lem:inc_no_loops} with \Cref{prop:no_loops}. In the standard case, each vertex has outdegree $0$ or $1$, but for increasing tableaux, there can be higher outdegrees since more than one number can move from row $i$ to row $i+1$; see \Cref{ex:inc_promotion_digraph}. 

Unlike the case for standard tableaux (see \Cref{cor:uniquely_determined}), it is possible for distinct increasing tableaux to share the same promotion digraph. In particular, while \Cref{thm:antiexcedance} makes it easy to read off the top row of a standard tableau $T$ from $\prom_1(T)$, the top row of an increasing tableau $U$ is not generally determined by $\prom_1(U)$.

\begin{example}\label{ex:cant_get_shape}
    Consider the tableaux below in the alphabet $1 < \dots < 5$.
    \medskip
    \[T_1=\ytableaushort{135,2}\hspace{.2in}T_2=\ytableaushort{145,2}\hspace{.2in}T_3=\ytableaushort{15,2}\hspace{.2in}T_4=\ytableaushort{125,2}\]
    They all share the following.
    \[
    \begin{tikzpicture}[scale=2, every node/.style={minimum size=3mm}]
    \node (label) at (-2.6,-0.5) {$\prom_1(T_1) =\prom_1(T_2) =\prom_1(T_3) =\prom_1(T_4) =$};
    \node (5) at (-0.5, -0.5) {5};   
    \node (1) at (0.3, -0.5) {1};  
    \node (2) at (1.1, -0.5) {2};  
    \node (3) at (1.9, -0.5) {3}; 
    \node (4) at (2.7, -0.5) {4}; 
    \draw[->, thick] (5) -- (1);
    \draw[->, thick] (1) -- (2);
\end{tikzpicture}
\]
Since the tableaux have only two rows, there are no other nontrivial promotion digraphs. 
\end{example}

\begin{example}\label{ex:cant_get_top_row}
    In the previous example, the tableaux have different sets of entries or have different shapes. However, even packed increasing tableaux of the same shape can have the same tuple of promotion digraphs. Let
\[  T=\ytableaushort{12,2,3}\rightarrow\ytableaushort{2{\mathbf{1}},3,{\mathbf{1}}}\rightarrow\ytableaushort{3{\mathbf{1}},{\mathbf{1}},{\mathbf{2}}}\rightarrow\ytableaushort{{\mathbf{1}}{\mathbf{3}},{\mathbf{2}},{\mathbf{3}}}
\]
and let
\[
T' = \ytableaushort{13,2,3}\rightarrow\ytableaushort{23,3,{\mathbf{1}}}\rightarrow\ytableaushort{3{\mathbf{2}},{\mathbf{1}},{\mathbf{2}}}\rightarrow\ytableaushort{{\mathbf{1}}{\mathbf{2}},{\mathbf{2}},{\mathbf{3}}}.
\]
Then we have the following promotion digraphs.
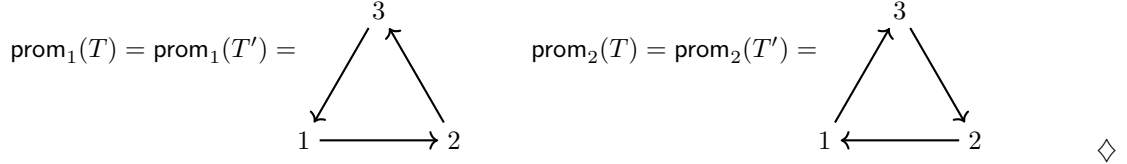
\[
\begin{tikzpicture}[scale=2, every node/.style={minimum size=3mm}]
    \node (label) at (-1.0,0.6) {$\prom_1(T) = \prom_1(T') =$};
    
    \node (1) at (0, 0) {1};             
    \node (2) at (1, 0) {2};             
    \node (3) at (0.5, {sqrt(3)/2}) {3}; 
    \draw[->, thick] (1) -- (2);
    \draw[->, thick] (2) -- (3);
    \draw[->, thick] (3) -- (1);
\end{tikzpicture}
\quad  \quad 
\begin{tikzpicture}[scale=2, every node/.style={minimum size=3mm}]
    \node (label) at (-1.0,0.6) {$\prom_2(T) = \prom_2(T') =$};
    \node (1) at (0, 0) {1};             
    \node (2) at (1, 0) {2};             
    \node (3) at (0.5, {sqrt(3)/2}) {3}; 
    \draw[<-, thick] (1) -- (2);
    \draw[<-, thick] (2) -- (3);
    \draw[<-, thick] (3) -- (1);
\end{tikzpicture} \qedhere
\]
\end{example}

By the preceding examples, promotion digraphs do not determine the top row of their corresponding tableau. However, they do determine the bottom row of the tableau. Indeed, we have the following. Recall the definition of excedances from \Cref{def:excedance}.

\begin{lemma}
\label{lem:bottom_i_rows}
    Suppose $T \in \inc^q(\lambda)$ where $\lambda$ has $r$ rows. Then $\ell$ appears in the union of the bottom $i$ rows if and only if $\ell$ is an excedance of $\prom_{r-i}(T)$.
\end{lemma}
\begin{proof}
    Suppose that $\ell$ appears in the union of the bottom $i$ rows in $T$, say in row $k$ (and possibly in other rows as well).
    The label $\ell$ appears in only the upper left corner of $\gromotion^{\ell-1}(T)$. Hence, $\ell$ must pass through each intervening row $2, 3, \dots, k-1$ during an earlier iteration of gromotion. In particular, there is an earlier iteration of gromotion where an $\ell$ moves from row $r-i+1$ to row $r-i$. Thus $\prom_{r-i}(T)$ has an incoming arrow from a vertex $j < \ell$, so $\ell$ is an excedance of $\prom_{r-i}(T)$.

    Conversely, if $\ell$ is an excedance of $\prom_{r-i}(T)$, then $\ell$ passed from row $r-i+1$ to row $r-i$ during an earlier iteration of gromotion. Therefore $\ell$ must have begun in the bottom $i$ rows in $T$.
\end{proof}

\begin{corollary}\label{cor:recover_bottom}
    Suppose $T \in \inc^q(\lambda)$ where $\lambda$ has $r$ rows. Then $\prom_{r-1}(T)$ determines the bottom row of~$T$. Specifically, the entries of the bottom row of $T$ are exactly the excedances of $\prom_{r-1}(T)$. \qed
\end{corollary}

\begin{corollary}\label{cor:new_in_row}
 Suppose $T \in \inc^q(\lambda)$. Then if vertex $\ell$ has positive indegree in any $\prom_i$, then the value $\ell$ appears in $T$.
If furthermore $\ell$ is a nonexcedance of $\prom_1(T)$, then $\ell$ appears in the first row of $T$ and in no other row.
\end{corollary}
\begin{proof}
If $\prom_i(T)$ has an arrow $j \to \ell$ for some $i$ and $j$, then some iteration of gromotion moves a label $\ell$ from row $i+1$ to row $i$. Since gromotion does not change the set of labels that appear, this means that $T$ has an instance of $\ell$.
By \Cref{lem:bottom_i_rows}, if $\ell$ is a nonexcedance of $\prom_1(T)$, then $\ell$ cannot appear outside the first row of $T$. 
\end{proof}

\begin{example}
Recall the tableau $T$ from \Cref{ex:inc_promotion_digraph} and its promotion digraphs. We see that the excedances of $\prom_2(T)$ are $4$, $6$, and $8$, which make up the bottom row of $T$ by \Cref{cor:recover_bottom}. The excedances of $\prom_1(T)$ are $2$, $3$, $4$, $6$, $8$, and $9$, meaning that the bottom two rows are $T$ are comprised of these entries by \Cref{lem:bottom_i_rows}. We then know that $2$, $3$, and $9$ appear in the second row of $T$---because we know they appear in the bottom two rows but not in the bottom row---but we do not know whether $4$, $6$, and/or $8$ appear in the second row. Since $1$ and $5$ are nonexcedances of $\prom_1(T)$ that have incoming arrow in $\prom_2(T)$, we know that these entries appear only in the top row of $T$ by \Cref{cor:new_in_row}.
\end{example}

In the case that $\lambda$ is a rectangle, we consider stronger versions of these results in the next subsection.

\subsection{Rectangular increasing tableaux}\label{sec:rect_inc}
$K$-promotion is best understood in the rectangular case. In this case, moreover, we can say much more about the structure of the associated promotion digraphs. In particular, we propose in \Cref{conj:inc_tab_determined} that a rectangular increasing tableau is completely determined by its tuple of promotion digraphs. We use the following two lemmas.


\begin{lemma}
\label{lem:next_row_up}
    Suppose $T \in \inc^q(r\times c)$.
    Suppose that value $v$ appears in the union of the bottom $k$ rows of $T$. 
    Let $t$ be minimal such that there is an arrow $t\to v$ in $\prom_{r-k}$. 
    Consider the following three conditions: 
    \begin{enumerate}
        \item[$(1)$] $v$ also appears in row $r-k$;
        \item[$(2)$] $\prom_{r-k-1}(T)$ has an arrow $t' \to v$ with $t' < t$;  
        \item[$(3)$] $\prom_{r-k}(T)$ has an arrow $t \rightarrow w$ with $v < w$.
    \end{enumerate}
    Then 
    \begin{itemize}
        \item condition $(1)$ implies that either condition $(2)$ or $(3)$ holds, \item condition $(2)$ implies condition $(1)$, and 
    \item condition $(3)$ implies condition $(1)$ if $k = r-1$.
    \end{itemize}
\end{lemma}
\begin{proof}
    Since $v$ is in the union of the bottom $k$ rows, $v$ is an excedance of $\prom_{r-k}(T)$ by \Cref{lem:bottom_i_rows}. Fix $t$ minimal such that there is an arrow $t \to v$ in $\prom_{r-k}(T)$. Note that $t < v$.

    Now suppose (1) holds, so that the value $v$ also appears in row $r-k$. Then one or both of the following two things happens. If this $v$ moves up to row $r-k-1$ before iteration $t$ of gromotion, then $\prom_{r-k-1}(T)$ has an arrow $t' \to v$ with $t' < t$ and so (2) holds. 
    Otherwise, the $v$ does not move up earlier and so iteration $t$ of gromotion involves the $v$ from the bottom $k$ rows merging with the $t$ in row $r-k$ so that the flow path of this iteration continues to the right in row $r-k$ and another value $w > v$ must also enter row $r-k$ during this gromotion.
    Thus $t \rightarrow w$ in $\prom_{r-k}(T)$ with $v < w$, and so (3) holds.

    Suppose instead that (2) holds. By definition of $t$, before iteration $t'$ of gromotion, no $v$ has moved into row $r-k$. But (2) implies that $v$ moves from row $r-k$ to row $r-k-1$ at iteration $t'$. Hence, $v$ appeared in row $r-k$ in the original tableau $T$.

    Finally, suppose that (3) holds with $k = r-1$ and fix $w$ satisfying the hypothesis. We wish to show that $T$ has $v$ also appearing in row $1$. By (3), iteration $t$ of gromotion transports both $v$ and $w$ from row $2$ to row $1$. Say that $v$ moves up in column $c$ while $w$ moves up in column $d > c$. Considering the flow path of this iteration, we see that it must contain the segment of the first row between columns $c$ and $d$. Since the flow path also contains the edge from $(2,c)$ to $(1,c)$, this implies that there must have been a merge into the cell $(1,c)$. Thus, $v$ was also in position $(1,c+1)$ before this iteration. By choice of $t$, $v$ must therefore have also been in row $1$ in the initial tableau $T$. 
\end{proof}

\begin{remark}
    One might hope that \Cref{lem:next_row_up} could instead say that $(3)$ implies $(1)$ for general $k$. The following example shows this does not hold. Let 
    \[
    T = \ytableaushort{12,24,35}
    \]
    and consider the value $v=3$.
    Then $\prom_{3-1}(T)$ has an arrow $1 \to 3$.  However, $\prom_2(T)$ also has an arrow $1 \to 5$. 
\end{remark}


\begin{lemma}\label{lem:inc_is_it_packed}
    Suppose $T \in \inc^q(r \times c)$ and $1 \leq i < r$. Then $T$ is packed if and only if every vertex in $\prom_i(T)$ has outdegree at least $1$.  Indeed, a vertex $v$ in $\prom_i(T)$ has outdegree $0$ exactly if the value $v$ does not appear in $T$.
    \end{lemma}
\begin{proof}
    If the value $v$ does not appear in $T$, then in the $v$th iteration of gromotion, no labels will move and so vertex $v$ will have outdegree $0$ in each $\prom_i(T)$. On the other hand, if $v$ appears in $T$, then the $v$th application of gromotion involves a nonempty flow path; since $T$ is rectangular, this path ends at the lower right corner $(r,c)$ and crosses every boundary between rows, so $v$ has nonzero outdegree in every $\prom_i(T)$.
\end{proof}

We also have a sort of dual statement to \Cref{lem:bottom_i_rows} in the rectangular case.

\begin{lemma}\label{lem:top_i_rows}
    Suppose $T \in \inc^q(r \times c)$ and $1 \leq i < r$. Then $\ell$ appears in the union of the top $i$ rows of $\promotion^q(T)$ if and only if $\ell$ is an antiexcedance of $\prom_{i}(T)$.
\end{lemma}
\begin{proof}
    The value $\ell$ can only leave the union of the top $i$ rows when it is located in the cell $(1,1)$. This only happens immediately before iteration $\ell$ of gromotion. After this iteration, $\ell$ appears only in the cell $(r,c)$. 
    
    If $\ell$ later enters the union of the top $i$ rows at iteration $m > \ell$ of gromotion, and in particular if $\ell$ appears in the union of the top $i$ rows of $\promotion^q(T)$, then $\prom_{i}(T)$ has an arrow $m \to \ell$ and $\ell$ is an antiexcedance of $\prom_{i}(T)$. Conversely, if $\ell$ is an antiexcedance of $\prom_{i}(T)$, then $\prom_{i}(T)$ has an arrow $m \to \ell$ for some $m > \ell$, and so $\ell$ enters the union of the top $i$ rows during iteration $m$ of gromotion; as previously described, $\ell$ cannot then leave this region until iteration $q + \ell$, so $\ell$ appears in the union of the top $i$ rows of $\promotion^q(T)$.
\end{proof}

Following \cite{Pechenik:frames}, define the \newword{frame} of a rectangular tableau to be the union of the first row, last row, first column, and last column of the tableau.

\begin{lemma}
\label{lem:frame}
    Suppose $T \in \inc^q(r\times c)$ and fix a value $v$. Then the promotion digraphs determine exactly which boxes of the frame of $T$ are occupied by $v$.   
\end{lemma}
\begin{proof}
    We can determine the bottom row of $T$ by \Cref{cor:recover_bottom}.
    We can determine the leftmost column because, by increasingness, the entry in position $(i,1)$ is the least element of the union of the bottom $r+1-i$ rows; we know this union by \Cref{lem:bottom_i_rows}.
    
    We can determine if entry $v$ is in row $1$ and also in a lower row, by using Condition (3) from \Cref{lem:next_row_up}. Entry $v$ is only in the top row of $T$ exactly when it appears somewhere in the tableau as determined by \Cref{lem:inc_is_it_packed}, but does not appear in the union of the bottom $r-1$ rows as determined by \Cref{lem:bottom_i_rows}. 

    To determine the rightmost column of $T$, first note that, by increasingness, the entry in position $(i,c)$ of $\promotion^q(T)$ is the greatest element of the union of the top $i$ rows of $\promotion^q(T)$, which is determined by \Cref{lem:top_i_rows}. By \cite[Theorem~2]{Pechenik:frames}, the rightmost column of $\promotion^q(T)$ coincides with the rightmost column of $T$.
\end{proof}

\begin{corollary}
    \Cref{conj:inc_tab_determined} holds for $\inc^q(2 \times c)$ and $\inc^q(r \times 2)$. 
\end{corollary}
\begin{proof}
    Immediate from \Cref{lem:frame}.
\end{proof}

\begin{example}
   Consider the tableau $U$ and its promotion digraphs from \Cref{ex:inc_promotion_digraph_rectangle}. Since $9$ is an excedance in $\prom_2(U)$, we know that $9$ is in the bottom row of $U$ by \Cref{cor:recover_bottom}. To determine if it also appears in the second row of $U$ using \Cref{lem:next_row_up}, we first notice that $1$ is the minimal value that has an arrow pointing to $9$ in $\prom_2(U)$. We see that $1$ also has an arrow pointing to $10$ in $\prom_2(U)$, and hence we conclude that $9$ appears in the second row of $U$ using the second bullet point of the Lemma.
\end{example}

Unlike the case for general increasing tableaux (see \Cref{ex:cant_get_shape,ex:cant_get_top_row}), we believe that a rectangular increasing tableau is completely determined by its tuple of promotion digraphs.
\begin{conjecture}\label{conj:inc_tab_determined}
    Suppose that $T,T'$ are rectangular increasing tableaux with the same tuples of promotion digraphs. Then $T = T'$. 
\end{conjecture}

We have verified this conjecture for computationally for hundreds of thousands of tableaux.
\Cref{conj:inc_tab_determined} would follow from an algorithm to reconstruct an increasing tableau $T$ from its tuple of promotion digraphs. Although we do not know such an algorithm in general, in practice, the partial results established so far are highly effective for reconstruction; see below for such an example. 

\begin{example}
A rectangular tableau $S$ has the promotion digraphs below. We  reconstruct $S$.
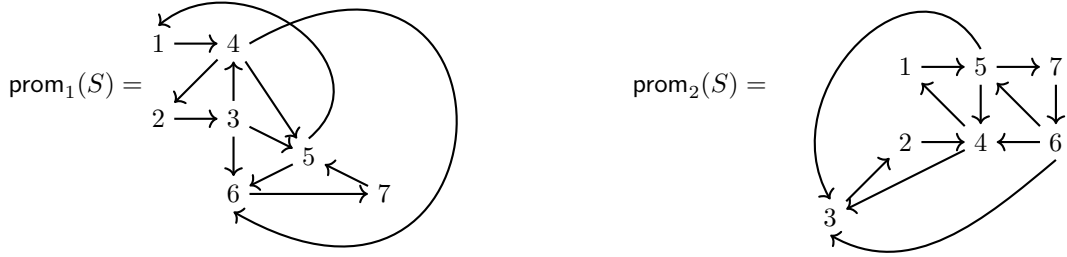
    \[\raisebox{.8in}{$\prom_1(S)=$}\begin{tikzpicture}
        \node (1) at (0,0) {$1$};
        \node (5) at (1,0) {$5$};
        \node (7) at (2,0) {$7$};
          \node (2) at (0,-1) {$2$};
        \node (4) at (1,-1) {$4$};
        \node (6) at (2,-1) {$6$};
        \node (3) at (-1,-2) {$3$};
        \draw[<-, thick] (1) to (5);
        \draw[<-, thick] (2) to (4);
        \draw[<-, thick] (3) to (2);
        \draw[<-, thick] (4) to (1);
        \draw[<-, thick] (4) to (3);
        \draw[<-, thick] (5) to (4);
        \draw[<-, thick] (5) to (7);
        \draw[<-, thick] (5.north) to [out=120,in=120, looseness=1.5] (3.north);
        \draw[<-, thick] (6.south) to [out=220,in=330, looseness=1] (3.south);
        \draw[<-, thick] (6) to (4);
        \draw[<-, thick] (6) to (5);
        \draw[<-, thick] (7) to (6);
    \end{tikzpicture}\hspace{.5in}
    \raisebox{0.8in}{$\prom_2(S)=$}\begin{tikzpicture}
        \node (1) at (0,0) {$1$};
        \node (5) at (1,0) {$5$};
        \node (7) at (2,0) {$7$};
          \node (2) at (0,-1) {$2$};
        \node (4) at (1,-1) {$4$};
        \node (6) at (2,-1) {$6$};
        \node (3) at (-1,-2) {$3$};
        \draw[->, thick] (1) to (5);
        \draw[->, thick] (2) to (4);
        \draw[->, thick] (3) to (2);
        \draw[->, thick] (4) to (1);
        \draw[->, thick] (4) to (3);
        \draw[->, thick] (5) to (4);
        \draw[->, thick] (5) to (7);
        \draw[->, thick] (5.north) to [out=120,in=120, looseness=1.5] (3.north);
        \draw[->, thick] (6.south) to [out=220,in=330, looseness=1] (3.south);
        \draw[->, thick] (6) to (4);
        \draw[->, thick] (6) to (5);
        \draw[->, thick] (7) to (6);
    \end{tikzpicture}\]
First, notice that each vertex of each promotion digraph has positive outdegree. This tells us that $S$ is packed with entries $1,2,\ldots,7$ by \Cref{lem:inc_is_it_packed}. Since $\prom_1(S)$ and $\prom_2(S)$ are the only nontrivial promotion digraphs, we know that $S$ has $3$ rows.
    
Next, we see that the excedances of $\prom_2(S)$ are $4$, $5$, and $7$. \Cref{cor:recover_bottom} tells us that these entries make up the bottom row of $S$. The excedances in $\prom_1(S)$ are $3$, $4$, $5$, $6$, and $7$. It follows that $3$ and $6$ must appear in the second row of $S$ by \Cref{lem:bottom_i_rows}. Using the first bullet point in \Cref{lem:next_row_up}, we can deduce that $4$ must also appear in the second row of $S$. Indeed, $4$ is an excedance in $\prom_2(S)$ with $2$ being the minimal vertex pointing to $4$, and there is an arrow from $1$ to $4$ in $\prom_1(S)$.

Since $1$ and $2$ are nonexcedances in $\prom_1(S)$, they appear only in the top row of $S$ by \Cref{cor:new_in_row}. We see that $5$ also appears in the top row using the second bullet point of \Cref{lem:next_row_up}: $3$ is the smallest entry with an arrow to $5$ in $\prom_1(S)$, and $3$ also has an arrow to $6$ in $\prom_1(S)$. We conclude that 
    \[S=\ytableaushort{125,346,457}. \qedhere \]
\end{example}

\begin{definition}\label{def:complete_digraph}
    A \newword{complete digraph} is a directed graph such that there is an arrow $u \to v$ for every ordered pair $(u,v)$ of distinct vertices.
\end{definition}

\begin{example}\label{ex:rect_nonrect_same_digraph}
The following rectangular and nonrectangular tableaux have the complete digraph on vertices $1, 2$ and $3$ as their $\prom_1$. 
\[\ytableaushort{12,23}\hspace{,5in} \ytableaushort{123,23} \qedhere\]
This shows that \Cref{conj:inc_tab_determined} cannot be strengthened to say that a rectangular tableau $T$ is distinguished from \emph{all} other tableaux by its promotion digraphs; we only conjecture that $T$ is distinguished from all other \emph{rectangular} tableaux.
\end{example}

\subsection{Vertex degrees in promotion digraphs of rectangular increasing tableaux}
\label{sec:vertex_degrees_inc}
We now prove a collection of results about the degrees of vertices in promotion digraphs of rectangular increasing tableaux. These results are analogous to the much easier \Cref{prop:no_loops} in the standard case.

\begin{lemma}\label{lem:same_row}
    Suppose $T \in \inc^q(r \times c)$ and $1 \leq k' \leq k \leq r$. If value $v$ appears in row $k$ of $T$ (perhaps among other rows) and appears in row $k'$ of $\promotion^q(T)$ (perhaps among other rows), then the indegree of vertex $v$ in $\prom_i(T)$ is at least $1$ for every $i$. 
\end{lemma}
\begin{proof}
        Suppose label $v$ exists in row $k$ of $T$. Then $v$ will appear only in the upper left-hand corner of $\gromotion^{v-1}(T)$, will appear only in the lower right-hand corner of $\gromotion^v(T)$, and will appear in the row $k'$ of $\gromotion^q(T)$.  Since $v$ passed through every row at least once, the lemma follows.
\end{proof}

\begin{corollary}\label{cor:indegree_smallorder}
    For any packed rectangular increasing tableau $T \in \packinc^q(r \times c)$, if $\promotion^q(T)=T$, then the indegree of each vertex in $\prom_i(T)$ is at least $1$. 
\end{corollary}
\begin{proof}
  This is immediate from the previous \Cref{lem:same_row}. 
\end{proof}

\begin{corollary}\label{cor:frame}
    For any packed rectangular increasing tableau $T \in \packinc^q(r \times c)$, the indegree of each vertex in each of $\prom_1(T)$ and $\prom_{r-1}(T)$ is at least $1$. Also, let $v$ be any value appearing in the frame of $T$. Then the indegree of $v$ in each $\prom_i$ is at least $1$.
\end{corollary}
\begin{proof}
For $v$ in the frame, this is immediate from \Cref{lem:same_row} combined with \cite[Theorem~2]{Pechenik:frames}, which states that $T$ and $\promotion^q(T)$ have the same frame. If $v$ appears only outside the frame, then it must pass through the first and last rows to return to being only outside the frame; this proves the statement about $\prom_1(T)$ and $\prom_{r-1}(T)$.
\end{proof}

\begin{corollary}\label{cor:row_r-1}
    Consider $T\in \packinc^q(r \times c)$, and let $v$ be any value appearing in row $r-1$ of $T$. The indegree of $v$ in any $\prom_i(T)$ is at least $1$.
\end{corollary}
\begin{proof} If $v$ also appears in the frame of $T$, we are done by \Cref{cor:frame}. Otherwise, we know $v$ appears only in the first row of $\gromotion^{v-1}(T)$, and only in the last row of $\gromotion^v(T)$. Since we know that the frame of $\promotion^q(T)$ is the same as the frame of $T$ by \cite[Theorem~2]{Pechenik:frames}, $v$ cannot be in row $r$ of $\gromotion^q(T)$. Thus the entry $v$ must reenter row $r-1$ at some point between the $v$th gromotion and the $q$th gromotion.
\end{proof}

\begin{example}\label{ex:indegree0}
Despite the previous sequence of results showing that vertices in promotion digraphs of packed rectangular increasing tableaux must have positive indegree in a large variety of cases, the indegree can in general be $0$.
For example, consider
\[T=
    \ytableaushort{135789{10}{13}{15}{18},2478{11}{14}{15}{17}{21}{22},68 {11} {16}{17}{20} {21}{22}{25}{26},8{10}{12}{17}{19}{22}{23}{24}{26}{27}} \in \packinc^{27}(4 \times 10).
    \]
    Here, the assiduous reader may check that $\prom_2(T)$ has indegree $0$ at vertex $14$. This means the $14$ does not make it back to row $2$ within $27$ iterations of gromotion. The promotion order of $T$ is $25 \cdot 27 = 675$. Examples of this behavior are challenging to locate. In particular, it is not feasible to search systematically, as for example $|\packinc^{27}(4 \times 10)| = 6,719,792,506,818,014,520$ by \cite[Theorem~2.4]{Pressey.Stokke.Visentin}.
\end{example}

The following shows that \Cref{ex:indegree0} is minimal with respect to its number of rows.

\begin{corollary}\label{cor:3row_indegree}
    Let $T\in \packinc^q(3 \times c)$. For each $i$, the indegree of each vertex $v$ is at least $1$.
\end{corollary}
\begin{proof}
    If $v$ appears in the frame of $T$, we are done by \Cref{cor:frame}. If not, then $v$ appears only in row $2$ and we are done by \Cref{cor:row_r-1}.
\end{proof}

\Cref{cor:3row_indegree} would also follow immediately from combining \Cref{cor:indegree_smallorder} and \Cref{conj:3row_order}; it may perhaps therefore be considered further weak evidence toward that conjecture.

\subsection{Summary of the preceeding results}\label{sec:summary}

Here, we summarize the various results obtained in \Cref{sec:inc_prom_digraph,sec:rect_inc,sec:vertex_degrees_inc} regarding the promotion digraphs of increasing tableaux. 

\begin{theorem}
\label{thm:inc_digraph_summary}
Let $T\in\inc^q(\lambda)$, where $\lambda$ has $r$ rows, and write
$D_i=\prom_i(T)$. We have the following.

\begin{enumerate}
\item Each digraph $D_i$ has no loops.

\item For $1\leq s<r$, the values appearing in the union of the bottom
$s$ rows of $T$ are exactly the excedances of $D_{r-s}$. In
particular, $D_{r-1}$ determines the bottom row of $T$.

\item If a value $v$ has nonzero indegree in some $D_i$, then $v$
appears in $T$. If, moreover, $v$ is a nonexcedance of $D_1$, then
$v$ appears only in the first row of $T$.
\end{enumerate}

Now assume that $\lambda=r\times c$ is rectangular. Then we additionally have the following.

\begin{enumerate}[resume]
\item For $1\leq i<r$, a vertex $v$ has outdegree $0$ in $D_i$
if and only if $v$ does not appear in $T$. Equivalently, $T$ is
packed if and only if every vertex of some (or equivalently every) $D_i$
has positive outdegree.

\item For $1\leq s<r$, the values appearing in the union of the top
$s$ rows of $\promotion^q(T)$ are exactly the antiexcedances of
$D_s$.

\item The tuple $(D_1,\ldots,D_{r-1})$ determines the frame of $T$.
Consequently, $(D_1,\ldots,D_{r-1})$ determines $T$ for
$\inc^q(2 \times c)$ and for $\inc^q(r \times 2)$. (Conjecturally, it determines $T$ for $\inc^q(r \times c)$.)

\item If $T$ is packed, then every vertex has nonzero indegree in
$D_1$ and $D_{r-1}$. Moreover, every value appearing in the frame of
$T$, and every value appearing in row $r-1$, has nonzero indegree in
every $D_i$. In particular, if $r = 2$ or $3$, every vertex of every $D_i$
has nonzero indegree.
\end{enumerate}
\end{theorem}
\begin{proof}
    This theorem collects together the statements of \Cref{lem:inc_no_loops,lem:bottom_i_rows,cor:new_in_row,lem:inc_is_it_packed,lem:top_i_rows,lem:frame,cor:frame,cor:row_r-1,cor:3row_indegree,conj:inc_tab_determined}.
\end{proof}

\section{Two-row and three-row increasing tableaux}\label{sec:small_r}
\subsection{Two-row increasing tableaux}\label{sec:2row}
The behavior of two-row rectangular increasing tableaux is much better understood than that of general rectangular increasing tableaux; see \cite{Pechenik:CSP,Bloom.Pechenik.Saracino,Kim.Rhoades,Kim:embedding,Patrias.Pechenik.Striker,Rhoades:skein}.
In particular, we have the following.

\begin{proposition}[{\cite{Pechenik:CSP}}]
    For $T \in \inc^q(2 \times c)$, we have $\promotion^q(T) = T$.
\end{proposition}

Our main result in this section is \Cref{thm:complete_graphs}, which completely characterizes the promotion digraphs of $2$-row rectangular tableaux. To give this characterization, we first must develop and recall some combinatorial ingredients.

We give below a first balance and teetering point characterization of promotion on $2$-row increasing tableaux (not necessarily rectangular). In this setting, we refer to $1$-balance and $1$-teetering points as simply balance and teetering points since there is no ambiguity.

\begin{theorem}\label{thm:Kpromotion_2row_via_balance_points}
   Let $T \in \inc^q(\lambda_1,\lambda_2)$ and let $w = w_1 \ldots w_q = \lattice(T)$. 
\begin{enumerate}
    \item If $w_1 = \varnothing$, then $\lattice(\promotion(T))$ is obtained from $w$ by deleting $w_1$ and appending $\varnothing$ to the end of $w$.
    \item If $w_1 \neq \varnothing$ and no balance point exists: 
    \begin{enumerate} 
        \item If $w$ has no teetering point, then $\lattice(\promotion(T))$ is obtained from $w$ by deleting the first letter of $w$, and appending a $1$ to the end.
        \item If $w$ has a teetering point, let $j_t$ be the first teetering point of $w$. Then $\lattice(\promotion(T))$ is obtained from $w$ by changing $w_{j_t}$ from $\{1,2\}$ to $1$, deleting the first letter of $w$, and appending a $\{1,2\}$ to the end. 
    \end{enumerate}
    \item If $w_1 \neq \varnothing$ and a balance point exists, let $j_b$ be the first balance point of $w$.
    \begin{enumerate}
        \item If there is no teetering point of $w$ before $j_b$, then $\lattice(\promotion(T))$ is obtained from $w$ by changing $w_{j_b}$ from $2$ to $1$, deleting the first letter of $w$, and appending a $2$ to the end
        \item If there is a teetering point of $w$ before $j_b$, let $j_t$ be the first teetering point of $w$; then $\lattice(\promotion(T))$ is obtained from $w$ by changing $w_{j_t}$ from $\{1,2\}$ to $1$, changing $w_{j_b}$ from $2$ to $\{1,2\}$, deleting the first letter of $w$, and appending a $2$ to the end.
    \end{enumerate}     
\end{enumerate} 
\end{theorem}

\begin{proof}
Case (1) is immediate from the definition.

For Cases (2) and (3), we consider the flow path of $T$.

If $w$ has no balance point and no teetering point (Case (2a)), then the flow path consists of the first row of $(\lambda_1,\lambda_2)$. Hence the lattice word changes as described.

If $w$ has no balance point but has a teetering point (Case (2b)), let $j_t$ be the first teetering point. In this case, $j_t$  is the least entry in the second row that is pointed to from above. In particular, $j_t$ in the second row is not pointed to from the left. However, $j_t$ also appears in the first row. There may be other teetering points $j_{o_1}, \dots, j_{o_\ell}$ after $j_t$; these appear on the flow path as the values in the second row that are pointed to both from above and from the left. 

The values that move from row $2$ to row $1$ during gromotion are the values $j_t, j_{o_1}, \dots, j_{o_\ell}$. The value $j_t$ is in both rows before gromotion, but only in row $1$ after gromotion. The values $j_{o_1}, \dots, j_{o_\ell}$ are in both rows both before and after gromotion. Value $1$ ends up in the rightmost box of both rows. All other entries do not change rows. Replacing the $1$ at the beginning of $\mathcal{{L}}(T)$ with $\{1,2\}$ at the end reflects the cyclic shift of the alphabet in promotion.

Now suppose that we are in Case (3), so that $w$ has a balance point.
    Then the first balance point is the greatest entry $j_b$ in the second row that is pointed to from above. There is a teetering point before the first balance point if and only if $j_b$ in the second row is also pointed to from the left. In this case, the first teetering point $j_t$ before $j_b$ is the least entry in the second row that is pointed to from above. In particular, $j_t$ in the second row is not pointed to from the left. There may be other teetering points $j_{o_1}, \dots, j_{o_\ell}$ between $j_t$ and $j_b$; these appear on the flow path as the intermediate values in the second row that are pointed to both from above and from the left. 

The values that move from row $2$ to row $1$ during gromotion are the values $j_t, j_{o_1}, \dots, j_{o_\ell}, j_b$. The value $j_t$ is in both rows before gromotion, but only in row $1$ after gromotion. The values $j_{o_1}, \dots, j_{o_\ell}$ are in both rows both before and after gromotion. The value $j_b$ is only in row $2$ before gromotion, but is in both rows afterwards if $j_t$ exists and only in the top row if $j_t$ does not exist. Value $1$ ends up in the rightmost box of the second row. All other entries do not change rows. Replacing the $1$ at the beginning of $\mathcal{{L}}(T)$ with $2$ at the end reflects the cyclic shift of the alphabet in promotion.
\end{proof}

\begin{example}
We illustrate an example of Case (2b) of \Cref{thm:Kpromotion_2row_via_balance_points}. Tableau $T$ and its promotion are shown below, along with the flow path for $T$.
\[T=\ytableaushort{123467,235}\hspace{.5in}\promotion(T)=\ytableaushort{123567,247}\]
\[\begin{tikzpicture}
    \node (1) at (0,0) {$\mathbf{1}$};
    \node (2) at (1,0) {$\mathbf{2}$};
    \node (3) at (2,0) {$\mathbf{3}$};
    \node (4) at (3,0) {$\mathbf{4}$};
    \node (6) at (4,0) {$\mathbf{6}$};
    \node (7) at (5,0) {$\mathbf{7}$};
    \node (2second) at (0,-1) {$\mathbf{2}$};
    \node (3second) at (1,-1) {$\mathbf{3}$};
    \node (5second) at (2,-1) {$\mathbf{5}$};
    \draw[noflow] (3) to (5second);
    \draw[flow] (1) to (2);
    \draw[flow] (2) to (3);
    \draw[flow] (3) to (4);
    \draw[flow] (4) to (6);
    \draw[flow] (6) to (7);
    \draw[flow] (1) to (2second);
    \draw[flow] (2) to (3second);
    \draw[flow] (2second) to (3second);
    \draw[flow] (3second) to (5second);
\end{tikzpicture}\]
We compute $\mathcal{L}(T)$ and note that it has two teetering points (the first colored red below and the second colored green) but no balance points. 
\begin{align*}\mathcal{L}(T) &= 1 \ \textcolor{red}{\raisebox{-1ex}{$\overset{\displaystyle \mathbf{1}}{\underline{\mathbf{2}}}$}} \ \textcolor{dark}{\raisebox{-1ex}{$\overset{\displaystyle \mathbf{1}}{\underline{\mathbf{2}}}$}} \ 1 \ 2 \ 1 \ 1   \\
\mathcal{L}(\promotion(T)) &= \hspace{2ex} 1 \ \raisebox{-1ex}{$\overset{\displaystyle 1}{2}$} \ 1 \ 2 \ 1 \ 1 \ \raisebox{-1ex}{$\overset{\displaystyle 1}{2}$} \qedhere \end{align*}
\end{example}

\begin{remark}
Note that if $T$ is a two-row rectangular tableau, then only cases (1) and (3) of \Cref{thm:Kpromotion_2row_via_balance_points} can occur, since $T$ must have at least one balance point.
\end{remark}

\begin{example}\label{ex:inc_rect_lattice_word} 
Consider the tableau $U$ and its promotion.
    \[U=\ytableaushort{1235689{12},34578{10}{11}{13}}\hspace{.5in}\promotion(U)=\ytableaushort{124578{10}{11},23679{10}{12}{13}}\]
The flow path of $U$ is below:
    \[\begin{tikzpicture}
        \node (1) at (0,0) {$\mathbf{1}$};
        \node (2) at (1,0) {$\mathbf{2}$};
        \node (3) at (2,0) {$\mathbf{3}$};
        \node (5) at (3,0) {$\mathbf{5}$};
        \node (6) at (4,0) {$\mathbf{6}$};
        \node (8) at (5,0) {$\mathbf{8}$};
        \node (9) at (6,0) {$\mathbf{9}$};
        \node (12) at (7,0)  {$\textcolor{gray}{12}$};
        \node (13) at (7,-1)  {$\mathbf{13}$};
        \node (3second) at (0,-1) {$\textcolor{gray}{3}$};
        \node (4second) at (1,-1) {$\textcolor{gray}{4}$};
        \node (5second) at (2,-1) {$\mathbf{5}$};
        \node (7second) at (3,-1) {$\mathbf{7}$};
        \node (8second) at (4,-1) {$\mathbf{8}$};
        \node (10second) at (5,-1) {$\mathbf{10}$};
        \node (11second) at (6,-1) {$\mathbf{11}$};
        \node (12) at (7,0) {\textcolor{gray}{$12$}};
        \draw[noflow] (1) to (3second) to (4second) to (2) to (4second) to (5second);
        \draw[noflow] (5) to (7second);
        \draw[flow] (1) to (2);
        \draw[flow] (2) to (3);
        \draw[flow] (3) to (5);
        \draw[flow] (5) to (6);
        \draw[flow] (6) to (8);
        \draw[flow] (8) to (9);
        \draw[flow] (9) to (11second);
        \draw[flow] (3) to (5second);
        \draw[flow] (5second) to (7second);
        \draw[flow] (7second) to (8second);
        \draw[flow] (6) to (8second);
        \draw[flow] (8second) to (10second);
        \draw[flow] (10second) to (11second);
        \draw[flow] (11second) to (13);
        \draw[noflow] (8) to (10second);
        \draw[noflow] (9) to (12) to (13);
        \end{tikzpicture}\]
We give below the lattice words of $U$ and $\promotion(U)$. The first teetering point of $U$ is shown in red, the first balance point in blue, and the teetering point appearing between these in green. Since we have a teetering point before the first balance point, we are in Case (3b) of \Cref{thm:Kpromotion_2row_via_balance_points}.
\begin{align*}\mathcal{L}(U) &=1 \ 1 \ \raisebox{-1ex}{$\overset{\displaystyle 1}{2}$} \ 2 \ \textcolor{red}{\raisebox{-1ex}{$\overset{\displaystyle \mathbf{1}}{\underline{\mathbf{2}}}$}} \ 1 \ 2 \ \textcolor{dark}{\raisebox{-1ex}{$\overset{\displaystyle \mathbf{1}}{\underline{\mathbf{2}}}$}} \ 1 \ 2 \ \textcolor{blue}{\underline{\mathbf{2}}} \ 1 \  2 \\
\mathcal{L}(\promotion(U)) &= \hspace{2ex} 1 \ \raisebox{-1ex}{$\overset{\displaystyle 1}{2}$} \ 2 \ 1 \ 1 \ 2 \ \raisebox{-1ex}{$\overset{\displaystyle 1}{2}$} \ 1 \ 2 \ \raisebox{-1ex}{$\overset{\displaystyle 1}{2}$} \ 1 \ 2  \ 2 \qedhere
\end{align*}
\end{example}

We now recall a construction from \cite{Pechenik:CSP}.
For $q$ a positive integer, a \newword{set partition} of $q$ is $\pi=\{\pi_1,\pi_2,\ldots,\pi_d\}$, where 
\begin{itemize}
    \item each $\pi_i \neq \emptyset$,
    \item $\bigcup_{i} \pi_i = [q]$, and
    \item $\pi_i \cap \pi_j = \emptyset$ if $i \neq j$.
\end{itemize}
The sets $\pi_i$ are called the \newword{blocks} of the set partition. 

We draw a set partition $\pi = \{\pi_1, \dots, \pi_d \}$ of $q$ by placing dots labeled $1, 2, \dots q$ clockwise around the boundary of a disk and then, for each $\pi_i$, drawing the convex hull of the boundary dots whose labels are in $\pi_i$. We call the set partition a \newword{noncrossing set partition} if these convex hulls do not intersect. We use $\nc(q)$ to denote the set of noncrossing set partitions of $q$.

\begin{lemma}[\cite{Pechenik:CSP}]\label{lem:equivariant_bij}
    Let $T \in \inc^q(2 \times c)$. There is a unique noncrossing set partition $\pi(T)$ of $[q]$ such that 
    \begin{itemize}
        \item $i \in [q]$ is a singleton block if and only if $i$ does not appear in $T$;
        \item $i$ is the least element of a nontrivial block if and only if $i$ appears only in the top row of $T$;
        \item $i$ is the greatest element of a nontrivial block if and only if $i$ appears only in the bottom row of $T$; and
        \item $i$ is neither least nor greatest in its block if $i$ appears in both rows of $T$.
    \end{itemize}
    The map
    \[\pi : \bigcup_c\inc^q(2 \times c) \to \nc(q)\]
    is a bijection intertwining tableau promotion with set partition rotation.
\end{lemma}
\begin{proof}
    The case of the bijection with packed tableaux and set partitions without singleton blocks is \cite[Proposition~2.3]{Pechenik:CSP}; the equivariance in that setting is \cite[Lemma 6.2]{Pechenik:CSP}. It is straightforward to see that one still has an equivariant bijection in the general case by making values $i$ that do not appear in $T$ correspond to singleton blocks. 
    
    Alternatively, one can obtain this lemma by composing the equivariant bijection of \cite[Lemma~4.2]{Dilks.Pechenik.Striker} with that of \cite[Theorem~7.8]{Striker.Williams}.
\end{proof}

We now connect the previous result to the idea of balance and teetering points.

\begin{lemma}\label{lem:teetering_and_balance_are_the_first_block}
    Let $T \in \inc^q(2 \times c)$, let $w = w_1 \ldots w_q = \lattice(T)$, and let $\pi = \pi(T)$. Write $B_1$ for the block of $\pi$ containing $1$. Then
    \begin{itemize}
        \item $w_1 = \varnothing$ if and only if $|B_1| = 1$;
        \item if $w_1 \neq \varnothing$, then the first balance point $j_b$ of $T$ equals $\max B_1$;
        \item if $w_1 \neq \varnothing$, then there is a teetering point before $j_b$ if and only if $|B_1| > 2$;
        \item if $w_1 \neq \varnothing$ and there is a teetering point before $j_b$, then these teetering points are exactly the elements of $B_1 \setminus \{ 1, j_b\}$; in particular, the first teetering point $j_t$ in this case equals $\min B_1 \setminus \{1\}$.
    \end{itemize}
\end{lemma}
\begin{proof}
The first bullet point is immediate by the first bullet point of \Cref{lem:equivariant_bij}. For the second bullet point, we have that $\max B_1$ is the smallest number such that $\{1, \dots \max B_1 \}$ contains equal numbers of least elements of blocks and greatest elements of blocks. By the bijection of \Cref{lem:equivariant_bij}, this means that $\max B_1$ is least such that each row of $T$ contains equal numbers of labels from $\{1, \dots \max B_1 \}$. By definition, this is the first balance point $j_b$ of $\lattice(T)$.
If $|B_1| \geq 2$, then the values $B_1 \setminus \{1, j_b \}$ are exactly the teetering points before $j_b$, establishing the third and fourth bullet points.
\end{proof}

We now have all the ingredients to prove our main result of this section. Recall \Cref{def:complete_digraph} for our convention on complete digraphs.

\begin{theorem}\label{thm:complete_graphs}
    Let $T\in\inc^q(2 \times c)$ and $\pi(T)$ its corresponding noncrossing set partition. Then $\prom_1(T)$ is a union of complete digraphs corresponding to the blocks of $\pi(T)$.
\end{theorem}
\begin{proof}
By \Cref{thm:Kpromotion_2row_via_balance_points}, $\prom_1(T)$ has an arrow directed from vertex $1$ to vertex $j_b$, where $j_b$ is the first balance point, and to each teetering point appearing before $j_b$. By \Cref{lem:teetering_and_balance_are_the_first_block}, these are exactly all the other elements in the same block as $1$ in $\pi(T)$.

By the equivariance of the map $\pi$, we therefore have that each vertex $v$ has arrows in $\prom_1(T)$ directed exactly to those other $w$ lying in the same block of $\pi(T)$ as $v$. Thus $\prom_1(T)$ is the desired union of complete digraphs.
\end{proof}

\begin{example}\label{ex:complete_digraph}
    Recall the tableau $U$ from \Cref{ex:inc_rect_lattice_word}. Then the corresponding set partition is 
    \[
    \begin{tikzpicture}[scale=2, every node/.style={circle, fill=black, inner sep=1pt}]
  \def\n{13}
  \def\r{1}
  \foreach \i in {1,...,\n} {
    \pgfmathsetmacro\angle{90 - 360*(\i-1)/\n}
    \coordinate (V\i) at ({\r*cos(\angle)}, {\r*sin(\angle)});
  }
  \draw[gray, thin] (0,0) circle(\r);
  \begin{scope}
    \fill[blue!10]
      (V2) to[bend right=20] (V3)
           to[bend right=20] (V4)
           to[bend left=20]  (V2);
  \end{scope}
  \begin{scope}
    \fill[blue!10]
      (V1) to[bend right=25] (V5)
           to[bend right=25] (V8)
           to[bend right=25] (V11)
           to[bend right=25] (V1);
  \end{scope}
  \foreach \i in {1,...,\n} {
    \node at (V\i) {};
    \node[fill=none, inner sep=0pt, anchor=center] at ($1.15*(V\i)$) {\i};
  }
  \draw[thick, blue] (V9) to[bend right=20] (V10);
  \draw[thick, blue] (V6) to[bend right=20] (V7);
  \draw[thick, blue] (V12) to[bend right=20] (V13);
  \draw[thick, blue] (V2) to[bend right=20] (V3);
  \draw[thick, blue] (V3) to[bend right=20] (V4);
  \draw[thick, blue] (V4) to[bend left=20]  (V2);
  \draw[thick, blue] (V1) to[bend right=25] (V5);
  \draw[thick, blue] (V5) to[bend right=25] (V8);
  \draw[thick, blue] (V8) to[bend right=25] (V11);
  \draw[thick, blue] (V11) to[bend right=25] (V1);
  \node[fill=none, anchor=east, xshift=-1.0cm] at (V11) {$\pi(U) =$};
\end{tikzpicture}
    \]
    We also have 
    \[
    \begin{tikzpicture}[scale=2, every node/.style={minimum size=3mm}]
  \node (label) at (-0.9, -0.6) {$\prom_1(U) =$};
  \node (1)  at (0, 0)     {1};
  \node (5)  at (1, 0)     {5};
  \node (8)  at (1, -1)    {8};
  \node (11) at (0, -1)    {11};
  \foreach \i in {1,5,8,11} {
    \foreach \j in {1,5,8,11} {
      \ifthenelse{\i=\j}{}{\draw[->, thick] (\i) -- (\j);}
    }
  }
  \node (2) at (2.4, 0)    {2};
  \node (3) at (2, -0.7)   {3};
  \node (4) at (2.8, -0.7) {4};
  \foreach \i in {2,3,4} {
    \foreach \j in {2,3,4} {
      \ifthenelse{\i=\j}{}{\draw[->, thick] (\i) -- (\j);}
    }
  }
  \node (6) at (0.3, 0.4) {6};
  \node (7) at (0.9, 0.4) {7};
  \draw[<->, thick] (6) -- (7);
  \node (9)  at (2.0, 0.4)  {9};
  \node (10) at (2.6, 0.4) {10};
  \draw[<->, thick] (9) -- (10);
  \node (12) at (1.1, -1.4) {12};
  \node (13) at (1.7, -1.4) {13};
  \draw[<->, thick] (12) -- (13);
\end{tikzpicture}
    \]
    in accordance with \Cref{thm:complete_graphs}.
\end{example}

\begin{remark}
    One could write a balance and teetering point characterization of promotion (the analogue of \Cref{thm:Kpromotion_2row_via_balance_points}) for increasing tableaux with more than two rows, but the characterization gets increasingly technical. As we currently have no application of such a formula, such as an analogue of \Cref{thm:complete_graphs}, we do not develop this here. Rather, in the next section, we discuss our motivating conjecture regarding 3-row increasing tableaux.
\end{remark}

\subsection{Trip digraphs and three-row increasing tableaux}\label{sec:trips}
We interpret the results of the prior subsection in terms of \emph{trip digraphs} of webs and give a motivating conjecture relating three-row increasing tableaux and their promotion digraphs to certain planar diagrams introduced recently by Jesse Kim \cite{Kim:embedding, Kim:flamingo}.

\begin{definition}[\cite{Postnikov}]
    A \newword{plabic graph} is a planar graph embedded in a disk such that all boundary vertices $b_1,b_2,\ldots,b_n$ have degree $1$ and each vertex is colored either black or white (not necessarily in a proper coloring). Plabic graphs are considered up to planar isotopy.
\end{definition}
    
\begin{definition}\label{def:trip_digraphs}
    For $1 \leq i \leq r-1$, the \newword{$(i,r)$-trip digraph} $\trip_{i,r}(G)$ of a plabic graph $G$ with minimum internal degree at least $r$ is defined as follows. 

    Starting at boundary vertex $b_{\alpha}$, travel along internal edges until another boundary vertex is reached. At each internal vertex $v$ of degree $h$ reached by edge $e$, if $v$ is white, leave $v$ by taking the $(i+\ell)$th left for some $0\leq \ell\leq h-r$, and  if $v$ is black, by taking the $(i+\ell)$th right for some $0\leq \ell\leq h-r$. If it is possible to reach boundary vertex $b_{\beta}$ by some sequence of allowed turn choices, then $\trip_{i,r}(G)$ has a directed edge $\alpha \to \beta$. 
\end{definition}

The previous \Cref{def:trip_digraphs} is new, although it is a variant of Postnikov's definition \cite{Postnikov} of a \emph{trip permutation} of a plabic graph. \Cref{def:trip_digraphs} recovers the functional digraph of Postnikov's trip permutation when the plabic graph $G$ has all internal degrees equal to $r$ and $i = 1$. When $G$ has all internal degrees equal to $r$ and $i$ is arbitrary, we obtain the functional digraphs of the trip permutations of \cite{Gaetz.Pechenik.Pfannerer.Striker.Swanson:SL4,Gaetz.Pechenik.Pfannerer.Striker.Swanson:2column}. When $G$ has an internal vertex of degree strictly greater than $r$, our digraph is significantly more complicated than any of these.

\begin{example}\label{ex:flamingo_trip_digraph}
We illustrate the trip digraphs for the plabic graph below with $r=3$.  

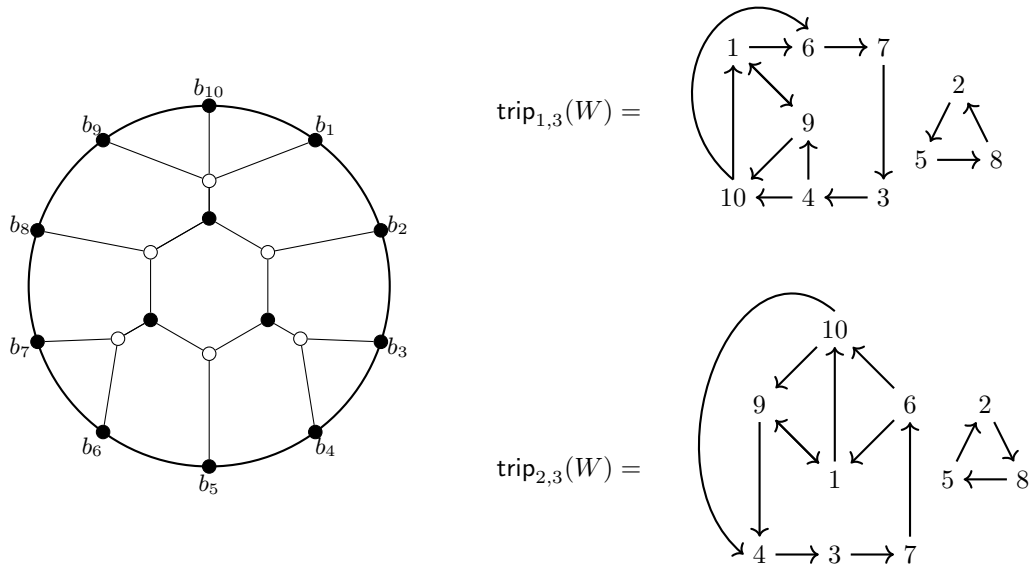
\begin{tikzpicture}[scale=2, every node/.style={font=\small}]
\tikzstyle{filled} = [circle, draw=black, fill=black, inner sep=1.8pt]
\tikzstyle{unfilled} = [circle, draw=black, fill=white, inner sep=1.8pt]
\def\outerradius{1.2}         
\def\innerradius{0.45}        
\def\outerwhiteradius{0.70}   
\def\labeloffset{0.12}        
\foreach \i in {1,...,10} {
  \node[filled] (v\i) at ({90-36*(\i-1)}:\outerradius) {};
  \ifnum\i=1
    \node at ({90-36*(\i-1)}:{\outerradius+\labeloffset}) {$b_{10}$};
  \else
    \node at ({90-36*(\i-1)}:{\outerradius+\labeloffset}) {$b_{\number\numexpr\i-1\relax}$};
  \fi
}
\draw[thick] (0,0) circle (\outerradius);
\foreach \i [count=\j from 1] in {90,150,210,270,330,30} {
  \ifodd\j
    \node[filled] (h\j) at (\i:\innerradius) {};
  \else
    \node[unfilled] (h\j) at (\i:\innerradius) {};
  \fi
}
\foreach \j in {1,...,6} {
  \pgfmathtruncatemacro{\k}{mod(\j,6)+1}
  \draw (h\j) -- (h\k);
}
\foreach \j/\angle in {1/90,3/210,5/330} {
  \node[unfilled] (w\j) at (\angle:\outerwhiteradius) {};
  \draw (h\j) -- (w\j); 
}
\draw (v1)--(w1)--(v2);
\draw (v10)--(w1)--(h1)--(h2)--(v9);
\draw (v8)--(w3)--(h3);
\draw (v7)--(w3);
\draw (v6)--(h4);
\draw (v5)--(w5)--(v4);
\draw (v3)--(h6);
\end{tikzpicture}\hspace{.3in}
\raisebox{1.2in}{$\begin{array}{l}
\raisebox{.5in}{$\trip_{1,3}(W)=$}\begin{tikzpicture}
        \node (1) at (0,0) {$1$};
        \node (6) at (1,0) {$6$};
        \node (7) at (2,0) {$7$};
        \node (9) at (1, -1) {$9$};
        \node (10) at (0,-2) {$10$};
        \node (4) at (1,-2) {$4$};
        \node (3) at (2,-2) {$3$};
        \node (2) at (3,-.5) {$2$};
        \node (5) at (2.5,-1.5) {$5$};
        \node (8) at (3.5,-1.5) {$8$};
        \draw[->, thick] (1) to (6);
        \draw[->, thick] (6) to (7);
        \draw[->, thick] (7) to (3);
        \draw[->, thick] (1) to (9);
        \draw[->, thick] (9) to (1);
        \draw[->, thick] (9) to (10);
        \draw[->, thick] (4) to (9);
        \draw[->, thick] (4) to (10);
        \draw[->, thick] (3) to (4);
        \draw[->, thick] (10) to (1);
        \draw[->, thick] (10.north) to [out=140,in=140, looseness=1.8] (6.north);
        \draw[->, thick] (5) to (8);
        \draw[->, thick] (8) to (2);
        \draw[->, thick] (2) to (5);
    \end{tikzpicture}\\
\raisebox{.5in}{$\trip_{2,3}(W)=$}
\begin{tikzpicture}
        \node (1) at (0,0) {$1$};
        \node (6) at (1,0) {$6$};
        \node (7) at (2,0) {$7$};
        \node (9) at (1, -1) {$9$};
        \node (10) at (0,-2) {$10$};
        \node (10b) at (-.1,-1.8) {};
        \node (4) at (1,-2) {$4$};
        \node (3) at (2,-2) {$3$};
        \node (2) at (3,-.5) {$2$};
        \node (5) at (2.5,-1.5) {$5$};
        \node (8) at (3.5,-1.5) {$8$};
        \draw[<-, thick] (1) to (6);
        \draw[<-, thick] (6) to (7);
        \draw[<-, thick] (7) to (3);
        \draw[<-, thick] (1) to (9);
        \draw[<-, thick] (9) to (1);
        \draw[<-, thick] (9) to (10);
        \draw[<-, thick] (4) to (9);
        \draw[<-, thick] (4) to (10);
        \draw[<-, thick] (3) to (4);
        \draw[<-, thick] (10) to (1);
        \draw[<-, thick] (10b.north) to [out=140,in=140, looseness=1.8] (6.north);
        \draw[<-, thick] (5) to (8);
        \draw[<-, thick] (8) to (2);
        \draw[<-, thick] (2) to (5);
    \end{tikzpicture}
\end{array}$}

The memorious reader may notice that these are the same digraphs as in \Cref{ex:inc_promotion_digraph_rectangle};  \Cref{conj:trip=prom,conj:inc_tab_determined} together propose an explanation for this phenomenon.

We now explain the calculation of the trip digraphs above. We see, for example, the outgoing arrows from $1$ in $\trip_{1,3}(W)$ coming from the paths in $W$ highlighted below on the left and the outgoing arrows from $1$ in $\trip_{2,3}(W)$ coming from the paths highlighted on the right. Note that in each case, after reaching a $4$-valent vertex, the path splits in two directions.

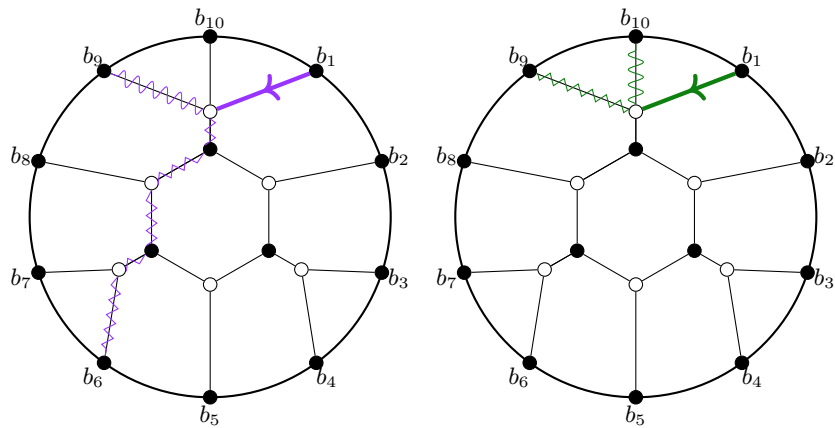
\begin{center}
\begin{tikzpicture}[scale=2, every node/.style={font=\small}]
\tikzstyle{filled} = [circle, draw=black, fill=black, inner sep=1.8pt]
\tikzstyle{unfilled} = [circle, draw=black, fill=white, inner sep=1.8pt]
\def\outerradius{1.2}         
\def\innerradius{0.45}        
\def\outerwhiteradius{0.70}   
\def\labeloffset{0.12}        
\foreach \i in {1,...,10} {
  \node[filled] (v\i) at ({90-36*(\i-1)}:\outerradius) {};
  \ifnum\i=1
    \node at ({90-36*(\i-1)}:{\outerradius+\labeloffset}) {$b_{10}$};
  \else
    \node at ({90-36*(\i-1)}:{\outerradius+\labeloffset}) {$b_{\number\numexpr\i-1\relax}$};
  \fi
}
\draw[thick] (0,0) circle (\outerradius);
\foreach \i [count=\j from 1] in {90,150,210,270,330,30} {
  \ifodd\j
    \node[filled] (h\j) at (\i:\innerradius) {};
  \else
    \node[unfilled] (h\j) at (\i:\innerradius) {};
  \fi
}
\foreach \j in {1,...,6} {
  \pgfmathtruncatemacro{\k}{mod(\j,6)+1}
  \draw (h\j) -- (h\k);
}
\foreach \j/\angle in {1/90,3/210,5/330} {
  \node[unfilled] (w\j) at (\angle:\outerwhiteradius) {};
  \draw (h\j) -- (w\j); 
}
\draw (v1)--(w1)--(v2);
\draw (v10)--(w1)--(h1)--(h2)--(v9);
\draw (v8)--(w3)--(h3);
\draw (v7)--(w3);
\draw (v6)--(h4);
\draw (v5)--(w5)--(v4);
\draw (v3)--(h6);
\draw[ultra thick, amethyst,oriented=>] (v2)--(w1);
\draw [decorate, decoration={zigzag, segment length=2mm, amplitude=.7mm}, amethyst] (w1)--(h1)--(h2)--(h3)--(w3)--(v7);
\draw[decorate, decoration={snake, segment length=2mm, amplitude=1mm}, amethyst] (w1)--(v10);
\end{tikzpicture}
\begin{tikzpicture}[scale=2, every node/.style={font=\small}]
\tikzstyle{filled} = [circle, draw=black, fill=black, inner sep=1.8pt]
\tikzstyle{unfilled} = [circle, draw=black, fill=white, inner sep=1.8pt]
\def\outerradius{1.2}         
\def\innerradius{0.45}        
\def\outerwhiteradius{0.70}   
\def\labeloffset{0.12}        
\foreach \i in {1,...,10} {
  \node[filled] (v\i) at ({90-36*(\i-1)}:\outerradius) {};
  \ifnum\i=1
    \node at ({90-36*(\i-1)}:{\outerradius+\labeloffset}) {$b_{10}$};
  \else
    \node at ({90-36*(\i-1)}:{\outerradius+\labeloffset}) {$b_{\number\numexpr\i-1\relax}$};
  \fi
}
\draw[thick] (0,0) circle (\outerradius);
\foreach \i [count=\j from 1] in {90,150,210,270,330,30} {
  \ifodd\j
    \node[filled] (h\j) at (\i:\innerradius) {};
  \else
    \node[unfilled] (h\j) at (\i:\innerradius) {};
  \fi
}
\foreach \j in {1,...,6} {
  \pgfmathtruncatemacro{\k}{mod(\j,6)+1}
  \draw (h\j) -- (h\k);
}
\foreach \j/\angle in {1/90,3/210,5/330} {
  \node[unfilled] (w\j) at (\angle:\outerwhiteradius) {};
  \draw (h\j) -- (w\j); 
}
\draw (v1)--(w1)--(v2);
\draw (v10)--(w1)--(h1)--(h2)--(v9);
\draw (v8)--(w3)--(h3);
\draw (v7)--(w3);
\draw (v6)--(h4);
\draw (v5)--(w5)--(v4);
\draw (v3)--(h6);
\draw [ultra thick, dark,oriented=>] (v2) -- (w1);
\draw [decorate, decoration={snake, segment length=2mm, amplitude=1mm}, dark] (w1)--(v1);
\draw[decorate, decoration={zigzag, segment length=2mm, amplitude=.7mm}, dark] (w1)--(v10);
\end{tikzpicture}
\end{center}

We see the outgoing arrow from $5$ in $\trip_{1,3}(W)$ coming from the path in $W$ highlighted below on the left and the outgoing arrow from $5$ in $\trip_{2,3}(W)$ coming from the path highlighted on the right. Note that if all the vertices visited by the trip are of degree $3$, the path never splits, and there is only one outgoing arrow on the trip digraph.

\[
\begin{tikzpicture}[scale=2, every node/.style={font=\small}]
\tikzstyle{filled} = [circle, draw=black, fill=black, inner sep=1.8pt]
\tikzstyle{unfilled} = [circle, draw=black, fill=white, inner sep=1.8pt]
\def\outerradius{1.2}         
\def\innerradius{0.45}        
\def\outerwhiteradius{0.70}   
\def\labeloffset{0.12}        
\foreach \i in {1,...,10} {
  \node[filled] (v\i) at ({90-36*(\i-1)}:\outerradius) {};
  \ifnum\i=1
    \node at ({90-36*(\i-1)}:{\outerradius+\labeloffset}) {$b_{10}$};
  \else
    \node at ({90-36*(\i-1)}:{\outerradius+\labeloffset}) {$b_{\number\numexpr\i-1\relax}$};
  \fi
}
\draw[thick] (0,0) circle (\outerradius);
\foreach \i [count=\j from 1] in {90,150,210,270,330,30} {
  \ifodd\j
    \node[filled] (h\j) at (\i:\innerradius) {};
  \else
    \node[unfilled] (h\j) at (\i:\innerradius) {};
  \fi
}
\foreach \j in {1,...,6} {
  \pgfmathtruncatemacro{\k}{mod(\j,6)+1}
  \draw (h\j) -- (h\k);
}
\foreach \j/\angle in {1/90,3/210,5/330} {
  \node[unfilled] (w\j) at (\angle:\outerwhiteradius) {};
  \draw (h\j) -- (w\j); 
}
\draw (v1)--(w1)--(v2);
\draw (v10)--(w1)--(h1)--(h2)--(v9);
\draw (v8)--(w3)--(h3);
\draw (v7)--(w3);
\draw (v6)--(h4);
\draw (v5)--(w5)--(v4);
\draw (v3)--(h6);
\draw [ultra thick, amethyst] (h4)--(h3)--(h2)--(v9);
\draw[ultra thick, amethyst,oriented=>] (v6) -- (h4);
\end{tikzpicture}
\begin{tikzpicture}[scale=2, every node/.style={font=\small}]
\tikzstyle{filled} = [circle, draw=black, fill=black, inner sep=1.8pt]
\tikzstyle{unfilled} = [circle, draw=black, fill=white, inner sep=1.8pt]
\def\outerradius{1.2}         
\def\innerradius{0.45}        
\def\outerwhiteradius{0.70}   
\def\labeloffset{0.12}        
\foreach \i in {1,...,10} {
  \node[filled] (v\i) at ({90-36*(\i-1)}:\outerradius) {};
  \ifnum\i=1
    \node at ({90-36*(\i-1)}:{\outerradius+\labeloffset}) {$b_{10}$};
  \else
    \node at ({90-36*(\i-1)}:{\outerradius+\labeloffset}) {$b_{\number\numexpr\i-1\relax}$};
  \fi
}
\draw[thick] (0,0) circle (\outerradius);
\foreach \i [count=\j from 1] in {90,150,210,270,330,30} {
  \ifodd\j
    \node[filled] (h\j) at (\i:\innerradius) {};
  \else
    \node[unfilled] (h\j) at (\i:\innerradius) {};
  \fi
}
\foreach \j in {1,...,6} {
  \pgfmathtruncatemacro{\k}{mod(\j,6)+1}
  \draw (h\j) -- (h\k);
}
\foreach \j/\angle in {1/90,3/210,5/330} {
  \node[unfilled] (w\j) at (\angle:\outerwhiteradius) {};
  \draw (h\j) -- (w\j); 
}
\draw (v1)--(w1)--(v2);
\draw (v10)--(w1)--(h1)--(h2)--(v9);
\draw (v8)--(w3)--(h3);
\draw (v7)--(w3);
\draw (v6)--(h4);
\draw (v5)--(w5)--(v4);
\draw (v3)--(h6);
\draw [ultra thick, dark] (h4)--(h5)--(h6)--(v3);
\draw[ultra thick, dark,oriented=>] (v6) -- (h4);
\end{tikzpicture} \qedhere
\]   
\end{example}

We now return to the discussion of noncrossing set partitions from \Cref{sec:2row} and show how to interpret the results there in terms of our new language of $(1,2)$-trip digraphs.

\begin{remark}\label{rem:set_part_2_plabic}
We turn the diagram of a noncrossing set partition $\pi\in\nc(q)$ into a plabic graph as follows. Consider each $i\in [q]$ as a black boundary vertex and replace each block $B$ of $\pi$ by a white interior vertex connected to all the black boundary vertices of $B$. 
For example, the noncrossing set partition of \Cref{ex:complete_digraph} becomes the plabic graph below.
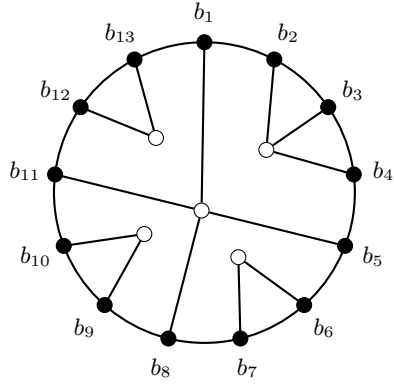
\[
\begin{tikzpicture}[scale=2, every node/.style={scale=1}]
  \def\n{13}
  \def\r{1}
  \foreach \i in {1,...,\n} {
    \pgfmathsetmacro\angle{90 - 360*(\i-1)/\n}
    \coordinate (V\i) at ({\r*cos(\angle)}, {\r*sin(\angle)});
   \pgfmathsetmacro\angle{90 - 360*(\i-1)/\n}
    \coordinate (W\i) at ({.5*\r*cos(\angle)}, {.5*\r*sin(\angle)});
  }
  \draw[thick] (0,0) circle(\r);
  \coordinate (T) at (W3);
  \node[draw, fill=white, circle, inner sep=2pt] (TT) at (T) {};
  \foreach \i in {2,3,4} {
    \draw[thick] (TT) -- (V\i);
  }
  \path (V5) -- (V11) coordinate[pos=0.33] (P1);
  \path  (V5) -- (V11) coordinate[pos=0.66] (P2);
  \path (P1) -- (P2) coordinate[midway] (T);
  \node[draw, fill=white, circle, inner sep=2pt] (SS) at (T) {};
  \foreach \i in {1,5,8,11} {
    \draw[thick] (SS) -- (V\i);
  }
  \path (V6) -- (V7) coordinate[pos=0.5] (M67);
  \path (M67) -- (0,0) coordinate[pos=0.5] (MM1);
  \draw[thick] (MM1) -- (V6);
  \draw[thick] (MM1) -- (V7);
    \node[draw, fill=white, circle, inner sep=2pt] at (MM1) {};
  \path (V9) -- (V10) coordinate[pos=0.5] (M910);
  \path (M910) -- (0,0) coordinate[pos=0.5] (MM2);
  \draw[thick] (MM2) -- (V9);
  \draw[thick] (MM2) -- (V10);
    \node[draw, fill=white, circle, inner sep=2pt] at (MM2) {};
  \path (V12) -- (V13) coordinate[pos=0.5] (M1213);
  \path (M1213) -- (0,0) coordinate[pos=0.5] (MM3);
  \draw[thick] (MM3) -- (V12);
  \draw[thick] (MM3) -- (V13);
    \node[draw, fill=white, circle, inner sep=2pt] at (MM3) {};
  \foreach \i in {1,...,\n} {
    \node[draw, fill=black, circle, inner sep=2pt] at (V\i) {};
    \node[font=\small] at ($1.2*(V\i)$) {$b_{\i}$};
  }
\end{tikzpicture}
\]
\end{remark}

With the interpretation from \Cref{rem:set_part_2_plabic} of noncrossing set partitions as plabic graphs, we have the following corollary to the results of \Cref{sec:2row}.
\begin{corollary}\label{cor:trip=prom_2row}
    The bijection $\pi$ between $2$-row rectangular increasing tableaux and noncrossing partitions satisfies $\trip_{1,2}(\pi(T))=\prom_1(T)$ for all $T \in \inc^q(2 \times c)$.
\end{corollary}
\begin{proof}
    This is straightforward from \Cref{thm:complete_graphs} and \Cref{def:trip_digraphs}.
\end{proof}

We now turn to the question of finding an analogue for $3$-row increasing tableaux, motivated by the conjecture \cite[Conjecture~4.12]{Dilks.Pechenik.Striker} on the order of $K$-promotion on $\inc^q(3\times c)$.
We recall below certain webs from \cite{Kim:embedding,Kim:flamingo}. 
We will first need the following definition, which is equivalent to \cite[Definition~7.5.7]{Fomin.Williams.Zelevinsky} after a color swap to match web conventions.

\begin{definition}
    A plabic graph $G$ is \newword{normal} if
    \begin{enumerate}
        \item all boundary vertices are black and have degree $1$,
        \item all internal black vertices have degree $3$,
        \item no two vertices of the same color are connected by an edge.
    \end{enumerate}
\end{definition}

Building on \cite{Fraser.Patrias.Pechenik.Striker}, which considered the special case of graphs with no internal black vertices, Jesse Kim \cite[\S 5]{Kim:embedding} introduced the following set of diagrams in conjectural connection to the Specht module $S^{(k,k,k,1^{n-3k})}$; Kim then established this connection in \cite{Kim:flamingo}. 

\begin{definition}[\cite{Kim:embedding}]
    A \newword{flamingo web} is a normal plabic graph where
    \begin{enumerate}
        \item every interior face has at least $6$ vertices and 
        \item every white vertex has degree at least $3$.
    \end{enumerate}
Let $\fw(n, k)$ denote the set of all flamingo webs with $n$ boundary vertices such that the number of white vertices minus the number of black interior vertices equals $k$.
\end{definition}

The following is our main conjecture.

\begin{conjecture}\label{conj:trip=prom}
    Each $W \in \fw(q,k)$ has a distinct pair $\trip_{1,3},\trip_{2,3}$ of trip digraphs. Moreover, for each $W$, there is a unique tableau $\tau(W) \in \inc^q(3 \times (q-2k))$ with $\trip_{i,3}(W) = \prom_i(\tau(W))$ for $i=1,2$. This yields an injection $\tau : \fw(q,k) \to \inc^q(3 \times (q-2k))$ intertwining web rotation with $K$-promotion.
\end{conjecture}

\begin{example}
The plabic graph $W$ from \Cref{ex:flamingo_trip_digraph}, reproduced below, is a flamingo web with $\tau(W)$ the tableau from \Cref{ex:inc_promotion_digraph_rectangle}. The reader may check that their promotion and trip digraphs are the same.

  \begin{center}
  \raisebox{1in}{$W=$}  
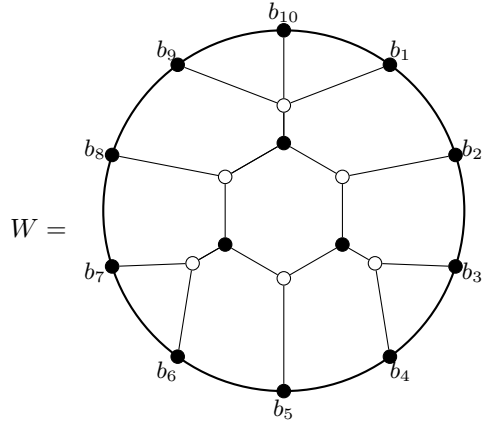
\begin{tikzpicture}[scale=2, every node/.style={font=\small}] 
\tikzstyle{filled} = [circle, draw=black, fill=black, inner sep=1.8pt]
\tikzstyle{unfilled} = [circle, draw=black, fill=white, inner sep=1.8pt]
\def\outerradius{1.2}         
\def\innerradius{0.45}        
\def\outerwhiteradius{0.70}   
\def\labeloffset{0.12}        
\foreach \i in {1,...,10} {
  \node[filled] (v\i) at ({90-36*(\i-1)}:\outerradius) {};
  \ifnum\i=1
    \node at ({90-36*(\i-1)}:{\outerradius+\labeloffset}) {$b_{10}$};
  \else
    \node at ({90-36*(\i-1)}:{\outerradius+\labeloffset}) {$b_{\number\numexpr\i-1\relax}$};
  \fi
}
\draw[thick] (0,0) circle (\outerradius);
\foreach \i [count=\j from 1] in {90,150,210,270,330,30} {
  \ifodd\j
    \node[filled] (h\j) at (\i:\innerradius) {};
  \else
    \node[unfilled] (h\j) at (\i:\innerradius) {};
  \fi
}
\foreach \j in {1,...,6} {
  \pgfmathtruncatemacro{\k}{mod(\j,6)+1}
  \draw (h\j) -- (h\k);
}
\foreach \j/\angle in {1/90,3/210,5/330} {
  \node[unfilled] (w\j) at (\angle:\outerwhiteradius) {};
  \draw (h\j) -- (w\j); 
}
\draw (v1)--(w1)--(v2);
\draw (v10)--(w1)--(h1)--(h2)--(v9);
\draw (v8)--(w3)--(h3);
\draw (v7)--(w3);
\draw (v6)--(h4);
\draw (v5)--(w5)--(v4);
\draw (v3)--(h6);
\end{tikzpicture} \hspace{1cm}
\ytableausetup{boxsize=1.5em}
\raisebox{1in}{$\tau(W)=\ytableaushort{1236,4569,789{10}}$}
\end{center}
\end{example}

\begin{remark}
    We believe that it should be possible to prove \Cref{conj:trip=prom} by developing a flamingo analogue of the separation labelings from \cite{Gaetz.Pechenik.Pfannerer.Striker.Swanson:SL4,Gaetz.Pechenik.Pfannerer.Striker.Swanson:2column}. Further discussion may appear elsewhere.
\end{remark}

\begin{remark}
    \Cref{conj:trip=prom} would imply currently unresolved cases of \cite[Conjecture~4.12]{Dilks.Pechenik.Striker} (stated in the introduction to this paper as \Cref{conj:3row_order}). It would be interesting to find a further generalization of flamingo webs corresponding to the remaining elements of $\inc^q(3 \times b)$. In particular, we currently do not know a corresponding construction for any of the cases where $q$ and $b$ have opposite parity. However, even in the case $q \equiv b \pmod 2$, the map $\tau$ of \Cref{conj:trip=prom} is generally not surjective.
\end{remark}

\begin{remark}
    It would also be interesting to develop flamingo analogues of the $\SL_4$-webs from \cite{Gaetz.Pechenik.Pfannerer.Striker.Swanson:SL4} and relate them to instances of $K$-promotion on $\inc^q(4 \times b)$. Here, an obvious obstacle (even beyond the complexity of the constructions of \cite{Gaetz.Pechenik.Pfannerer.Striker.Swanson:SL4}) is that the order of $K$-promotion on $\inc^q(4 \times b)$ is generally greater than $q$ (see e.g.~\Cref{ex:indegree0}). 
\end{remark}

\section*{Acknowledgements}
OP acknowledges support from NSERC Discovery Grant RGPIN-2021-02391 and Launch Supplement DGECR-2021-00010. JS acknowledges support from Simons Foundation gift MP-TSM-00002802 and NSF grant DMS-2247089.

We are grateful for helpful conversations with Ashleigh Adams, Ron Cherny, Christian Gaetz, Jesse Kim, Stephan Pfannerer, Sophie Spirkl, and Josh Swanson. We would also like to thank Ben Adenbaum for code to produce promotion digraphs. We are very grateful to the anonymous referees for their valuable comments, including identifying a false lemma in an earlier version of the manuscript.

\bibliographystyle{amsalphavar}
\bibliography{increasingwebs}
\end{document}